\documentclass[11pt,a4paper]{article}
\usepackage[OT1]{fontenc}
\usepackage[all]{xy}

\usepackage[english]{babel}
\usepackage{amsmath}

\usepackage{amsfonts}
\usepackage{amsthm}

\usepackage{graphicx}

\def\ss{\mathbb{S}}

\def\d{ {\cal D} }

\def\a{ {\cal A} }
\def\h{ {\cal H} }

\def\b{ {\cal B} }
\def\u{ {\cal U} }

\def\k{ {\cal K} }

\def\oo{ {\cal O} }
\def\vecx{ \left(\begin{array}{c} x_1 \\ x_2 \end{array} \right) } 
\def\vecy{ \left(\begin{array}{c} y_1 \\ y_2 \end{array} \right) } 
\def\vecz{ \left(\begin{array}{c} z_1 \\ z_2 \end{array} \right) } 
\def\q{ {\cal Q} }
\def\xx{ {\bf x} }
\def\yy{ {\bf y} }
\def\zz{ {\bf z} }

\def\vv{ {\bf v} }
\def\tr{ {\bf tr} }

\def\Exp{ {\bf {\rm Exp}}}

\def\Log{ {\bf {\rm Log}}}

\def\bi{{\bf \langle}}
\def\bd{{\bf \rangle}}

\def\pa{\a\mathbb{P}_1}

\def\g1{ \mathfrak{g}_1  }

\newtheorem{teo}{Theorem}[section]
\newtheorem{prop}[teo]{Proposition}
\newtheorem{lem}[teo]{Lemma}
\newtheorem{coro}[teo]{Corollary}
\newtheorem{defi}[teo]{Definition}
\theoremstyle{definition}
\newtheorem{rem}[teo]{Remark}
\newtheorem{ejem}[teo]{Example}

\title{Projective geometry in the Poincar\'e disk of  a  $C^*$-algebra}
\author{E. Andruchow, G. Corach, L. Recht}

\begin{document}

\maketitle 

\begin{abstract}
We study the Poincar\'e disk $\d=\{a\in\a: \|a\|<1\}$ of a C$^*$-algebra $\a$ from a projective point of view: $\d$ is regarded as an open subset of the projective line $\pa$, the space of complemented rank one submodules of $\a^2$. We introduce the concept of cross ratio of four points in $\pa$. Our main result establishes the relation between the exponential map $\Exp_{z_0}(z_1)$ of $\d$ ($z_0,z_1\in\d$) and the cross ratio of the four-tuple 
$$
\delta(-\infty), \delta(0)=z_0, \delta(1)=z_1 , \delta(+\infty),
$$ 
where $\delta$ is the unique geodesic of $\d$ joining $z_0$ and $z_1$ at times $t=0$ and $t=1$, respectively.
\end{abstract}
\bigskip

{\bf 2010 MSC:}  46L05, 58B20, 22E65, 46L08

{\bf Keywords:}  Projective line, Poincar\'e  disk, C$^*$-algebra.

\section{Introduction}

Let $\a$ be a unital C$^*$-algebra and consider the Poincar\'e disk 
$$
\d=\{z\in\a: \|z\|<1\}
$$ 
and the Poincar\'e halfspace 
$$
\h=\{ h\in\a: \frac{1}{2i}(h-h^*) \hbox{ is positive and invertible}\}.
$$
This paper is the second part of the study began in \cite{tejemas}, where we considered  $\d$ and  $\h$ as homogeneous spaces of the unitary groups $\u(\theta_\d)$ and $\u(\theta_\h)$ of certain quadratic forms $\theta_\d$ and $\theta_\h$ in $\a^2$, respectively. These spaces have reductive connections and invariant  Finsler metrics, which make them non positively curved length spaces (see \cite{gromov}).

In the present paper we study the  $\d$  from the projective point of view. More precisely, we introduce the projective line $\pa$ of $\a$, and regard $\d$ as an open subset of $\pa$. In this way, the group $\u(\theta)$ appears as the group of projective transformations which preserve $\d$. This allows us to introduce the concept of {\it cross ratio} of four points in $\pa$, in the sense considered by Zelikin \cite{zelikin}, adapted to our situation. Geometrically, the construction of the cross ratio is illustrated by the following figure:
	\begin{center}
	\includegraphics[width=0.5\textwidth]{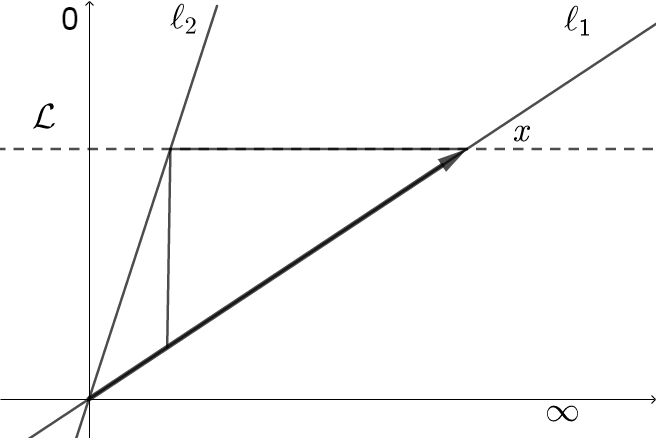}
	\end{center}
\centerline{\bf Figure 1.}
\bigskip

Here the plane represents $\a^2$ and the lines are elements of $\pa$, i.e. complemented submodules of $\a$ generated by a single element. The cross ratio of $\infty$, $0$, $\ell_1$ and $\ell_2$  (in this precise order) has the following geometric meaning: in the projective line $\pa$ we choose $\infty$ as the point of infinity, then $\pa\setminus\{\infty\}$ turns into an {\it affine line} (${\cal L}$ in the figure above). In this affine line, we choose a point $0$, and the affine line turns into a {\it vector line}. In this vector line we choose a point $\ell_1$, and the vector line turns into a {\it scalar  line} (identified with $\a$ by means of the basis $\ell_1$). Finally, if in this scalar line we choose a fourth element $\ell_2\ne 0$, this point has a coordinate in the basis $\ell_1$. Classically, this scalar is the cross ratio of $\infty, 0, \ell_1, \ell_2$. These features can be observed in Figure 1, in the line ${\cal L}$. In this paper, the cross ratio of these four lines will be an endomorphism $\varphi$  of $\ell_1$. The map $\varphi$ associates to each $\xx\in\ell_1$ the point $\varphi(\xx)$, obtained by two succesive projections: first to $0$, parallel to $\infty$; next to $\ell_1$, parallel to $\ell_2$. This is clear in the figure above.

We introduce a fibre bundle $\Gamma$ whose fibers are C$^*$-algebras (isomorphic to $\a$), over the base space $\d$. In  the tangent bundle $T\d$, we define  an inner product with values in $\Gamma$ (similarly as a Riemannian structure). If $z\in\d$ and $X,Y\in(T\d)_z$, we denote this inner product by $\langle X,Y\rangle_z$. It is an element in $\Gamma_z$ (the fiber of $\Gamma$ over $z$). Then, by definition, the cross ratio of four points in $\d$ is an element of $\Gamma$.

On the other hand, we compute explicitly the exponential map of the natural connection over $T\d$, which is bijective at every tangent space, and its inverse
$\Log_z:\d\to (T\d)_z$. We consider the morphism $\langle \Log_z(w), \Log_z(w)\rangle_z^{1/2}$ as a {\it distance operator} from $z$ to $w$ in $\d$.

In Theorem \ref{el teo} and Corollary \ref{el coro} we state  our main result, which clarifies the relationship between $\langle \Log_z(w), \Log_z(w)\rangle_z^{1/2}$ and the cross ratio  of $-\infty, z, w, \infty$. Here $\pm\infty$ denote the limits
$$
{\rm SOT}-\lim_{t\to\pm\infty}  \gamma(t),
$$
where $\gamma$ is the geodesic from $z$ to $w$ in time $1$, and the algebra $\a$ is considered in its universal representation.  The formulas in Theorem \ref{el teo} and Corollary \ref{el coro}  enable one to view the differential geometry of the Poincar\'e disk $\d$ as projective geometry of spaces of operators.

The contents of the paper are the following. In Section 2 we introduce the {\it regular} elements of $\a^2$, which will be used to define $\pa$ and $\pa^\theta$. 

In Section 3 we define the projective $\pa$  line of $\a$, and introduce the operator cross ratio  of four submodules (following ideas introduced in  \cite{zelikin} for closed subpaces): it is defined as a (possibly empty) set of module homomorphisms. 

In Section 4 we introduce our object of study, the $\theta$-hyperbolic part $\pa^\theta$ of $\pa$. We state the identifications of $\pa^\theta$ with the disk $\d$ and with the space of projections $\q_\rho$ ($\theta$-orthogonal projections whose ranges are the elements of $\pa^\theta$). We introduce the action of the $\theta$-unitary group $\u(\theta)$, as well as the Borel subgroup $\b(\theta)$.  

In Section 5 we describe the tangent bundle of $\pa^\theta$. 

In Section 6 we recall the basic facts of the geometry of $\d$ done in \cite{tejemas}. Also we compute the explicit form of the exponential map and its inverse $\Log$. 

In Section 7 we compute the limit points of geodesics at $t=\pm\infty$. In particular, we show that the partial isometry of $z\in\d$ coincides with the ${\rm SOT}-\lim_{t\to\infty} \delta(t)$, where $\delta$ is the geodesic joining $0$ and $z$.  

In Section 8 we introduce the $\u(\theta)$-invariant Finsler metric in $\pa^\theta$. 

In Section 9 we consider the cross ratio of the points $\ell_{-\infty}, \ell_0, \ell, \ell_{+\infty}$, where $\ell_0$ is the submodule corresponding to the origin in $\d$, $\ell$ is an arbitrary point in $\ell\ne\ell_0$, and $\ell_{\mp\infty}$ are the $\mp\infty$- limit points of the geodesic starting at $\ell_0$ at $t=0$ and reaching $\ell$ at $t=1$. We show that this set is always non empty, and that there is a natural choice of a specific homomorphism $cr(\ell_0,\ell)$, or $cr(0,z)$ if $\ell\simeq z\in\d$ (which is the only possible choice for a strongly dense subset of $\ell\in\pa^\theta$, in the case when $\a$ is a von Neumann algebra). We finish Section 9 by giving a first glimpse of our main result, which relates $cr(0,z)$ with the metric geometry of $\d$ (or equivalently, $\pa^\theta$), namely, 
$$
\|cr(0,z)\|_{\b(\ell_z)}=\frac12  d(0,z).
$$
where $\|cr(0,z)\|_{\b(\ell_z)}$ stands for the norm of the endomorphism $cr(0,z):\ell_z\to\ell_z$.  

In Section 10 we briefly introduce the coefficient bundle over $\pa^\theta$ (using the model $\q_\rho$), and an Hermitian structure on $\pa^\theta$, with values in the coefficient bundle. This is done in order to state our main results of this paper, which are the formulas given in Theorem  \ref{el teo} and Corollary \ref{el coro}, relating the cross ratio $cr(\ell_1,\ell_2)$  with the logarithm map $\Log_{\ell_1}\ell_2$. 

In Section 11 we consider a relevant example, namely, when there exists a C$^*$-subalgebra $\b$ of the center of $\a$ and a  conditional expectation $\tr:\a\to\b$ satisfying $\tr(xy)=\tr(yx)$ for all $x,y\in\a$; an important particular case is when $\a$ is commutative. In this case the computations are much simpler, and the   formula relating the cross ratio and the logarithm is
$$ 
e^{|\Log_{\ell_1} \ell_2|}=cr(\ell_1,\ell_2).
$$
\section{Preliminaries}
Let $\a$ be a unital C$^*$-algebra, $G$ the group of invertible elements of $\a$, $G^+$ the subset of $G$ of positive elements.  
Consider $\a^2=\a\times\a$ as a right $\a$-C$^*$-module, with $\a$-valued inner product given by
$$
\langle \xx,\yy \rangle= x_1^*y_1+x_2^*y_2,
$$
where $\xx=\vecx, \ \yy=\vecy\in\a^2$. $\a^2$ is endowed with the usual C$^*$-module norm
$$
\|\xx\|=\|<\xx,\xx>\|^{1/2}=\|x_1^*x_1+x_2^*x_2\|^{1/2}.
$$
 We shall say that an element $\xx=\vecx\in\a^2$ is {\it regular } if
$$
\langle \xx,  \xx \rangle\in G^+.
$$
We shall denote by $\a^2_{reg}$ the set of regular elements. Clearly $\a^2_{reg}$ is an open subset of $\a^2$.
The C$^*$-algebra of adjointable operators of $\a^2$ (bounded $\a$-linear operators acting in $\a^2$ which admit an adjoint for the $\a$-valued inner product) identifies with the algebra $M_2(\a)$ of $2\times 2$ matrices with entries in $\a$, acting by left multiplication.
Let $Gl_2(\a)$ be  the group of invertible elements of $M_2(\a)$.
Let us state the following basic facts:
\begin{prop}
With the current notations, we have that
\begin{enumerate}
\item
$\a^2_{reg}$ is an open subset of $\a$.
\item
$Gl_2(\a)$ acts on $\a^2_{reg}$ by left multiplication.
\item
$G$ acts on $\a^2_{reg}$ by right multiplication.
\item
Both actions commute.
\item
For each $\xx\in\a^2_{reg}$ the orbit $Gl_2(\a)\cdot \xx$ is open (and closed) in $\a^2_{reg}$.
\item
For each $\xx\in\a^2_{reg}$, the map $Gl_2(\a)\to GL_2(\a)\cdot \xx\subset \a^2_{reg}$, $\tilde{g} \mapsto \tilde{g}\cdot \xx$ defined a principal fibre bundle with smooth local cross sections.
\end{enumerate}
\end{prop}
We shall omit the  proof:  the proposition   holds  for  $\a^n$ instead of $\a^2$. We refer to \cite{cl} (Section 1) for complete proofs of these facts, noticing that 
$$
\a^n_{reg}=\{\xx\in\a^n: \hbox{ there exists } \yy\in\a^n \hbox{ such that } y_1x_1+\dots + y_nx_n=1\}.
$$
Indeed, clearly if $\xx\in\a^n_{reg}$, then putting $\yy$ given by $y_k=<\xx,\xx>^{-1}x_k^*$ provides an $n$-tuple such that $y_1x_1+\dots + y_nx_n=1$. Conversely, if $y_1x_1+\dots + y_nx_n=1$, and we denote by $\yy^*=(y_1^*,\dots,y_n^*)$, by the Schwarz inequality for C$^*$-modules (see for instance \cite{mt}, Prop. 1.2.4), we have
$$
1=<\yy^*,\xx>=|<\yy^*,\xx>|^2\le \|<\yy^*,\yy^*>\|<\xx,\xx>,
$$
which implies that $<\xx,\xx>$ is invertible.

This set $\a^n_{reg}$ is known, in the $K$-theoretic setting, as the set of $n$-{\it unimodular rows} in $\a$ . The paper by M. Rieffel \cite{rieffel} contains a thorough treatment of this subject.

We are interested in the single orbit 
$$
\oo:=Gl_2(\a)\cdot {\bf e_1},
$$
where ${\bf e_1}=\left( \begin{array}{c} 1 \\ 0 \end{array} \right)$.

Here and throughout, denote  $\hat{\tilde{g}}:=(\tilde{g}^*)^{-1}$.  

We consider the sphere
$$
\ss=\ss(\a^2)=\{\xx\in\a^2: \langle\xx,\xx\rangle=1\},
$$
and the intersection
$$
\ss_\oo=\ss\cap\oo.
$$

The next result  will not be needed in the rest of the paper, we state  it here to describe the topology of   the sphere $\ss$ in the case when $\a=\b(\h)$. 

\begin{teo}
If $\a=\b(\h)$, for an infinite dimensional Hilbert space $\h$, then:
\begin{enumerate}
\item
Each $\xx\in\ss$ defines an isometry $\xx:\h\to \h^2$ such that $R(\xx)^\perp$ is infinite dimensional.
\item
$\ss$ is connected.
\end{enumerate} 
\end{teo}
\begin{proof}
$$
\|\xx\xi\|^2=\|x_1\xi\|^2+\|x_2\xi\|^2=<x_1^*x_1\xi, \xi>+<x_2^*x_2\xi,\xi>=<\xi,\xi>=\|\xi\|^2.
$$
Let us check that $R(\xx)^\perp$ is infinite dimensional. Note that
$$
R(\xx)^\perp=\{(\xi,\eta)\in\h\times\h: x_1^*\xi=-x_2^*\eta\}.
$$
If either $N(x_1^*)$ or $N(x_2^*)$ is infinite dimensional, it is clear that $R(\xx)^\perp$ is infinite dimensional: suppose 
that $\dim N(x_1^*)=+\infty$; then all pairs of the form $(\xi, 0)$, with $\xi\in N(x_1^*)$, belong to $R(\xx)^\perp$.

Suppose then that both $N(x_1^*), N(x_2^*)$ are finite dimensional. Clearly 
$$
\{(x_1\varphi, x_2\psi): x_1^*x_1\varphi=-x_2^*x_2\psi\}\subset R(\xx)^\perp.
$$
Since $x_2^*x_2=1-x_1^*x_1$, the condition $x_1x_1^*\varphi=-x_2^*x_2\psi$ is equivalent to $\varphi=x_1^*x_1(\varphi-\psi)$. If we denote $\nu=\varphi-\psi$, then the left-hand set above can be written
$$
\{(x_1\varphi, x_2(\varphi+\nu)): \varphi=x_1^*x_1 \nu\}=\{(x_1x_1^*x_1\nu, x_2(x_1^*x_1\nu+\nu)): \nu\in\h\}.
$$
Note that the $R(x_1x_1^*x_1)$ is infinite dimensional: its orthogonal complement $R(x_1x_1^*x_1)^\perp$ satisfies
$$
R(x_1x_1^*x_1)^\perp=N(x_1^*x_1x_1^*)\subset N((x_1x_1^*)^2)=N(x_1x_1^*)=N(x_1^*),
$$
and thus is finite dimensional. Therefore $\xx$ is an isometry with infinite co-rank.

Let us check now that the sphere $\ss$ is connected. 
Consider a fixed unitary operator $U:\h\times\h\to\h$. Let $\xx, \yy\in\ss$, again, be regarded as isometries from $\h$ to $\h^2$. Then  the operators $U\xx$ and $U\yy$ are isometries in $\h$, with infinite co-rank. In \cite{arv} it was shown that the connected components of the set of isometries in $\h$ is parametrized by the co-rank: two isometries belong to the same component if and only if they have the same co-rank. Moreover, also in \cite{arv}, it was shown that these components are the orbits of the left action of the unitary group of $\h$, by left multiplication. It follows that there exists a selfadjoint operator $Z\in\b(\h)$ such that 
$$
U\yy=e^{iZ}U\xx.
$$
Then $\xx(t)=U^*e^{itZ}U\xx$ is a continuous  curve of isometries from $\h$ to $\h\times\h$ such that $\xx(0)=\xx$ and $\xx(1)=\yy$. Then
$$
\xx(t)=\left(\begin{array}{c} x_1(t) \\ x_2(t) \end{array} \right) \in\ss
$$
is a continuous curve in $\ss$ which connects $\xx$ and $\yy$.
\end{proof}

Let us denote by $\u_2(\a)$ the unitary group of $M_2(\a)$. Since $M_2(\a)$ is the C$^*$-algebra of adjointable operators of the module $\a^2$, $\u_2(\a)$ is the group of invertible elements in $M_2(\a)$ which preserve the $\a$-valued inner product of $\a^2$. In particular, this implies that $\u_2(\a)$ acts on $\ss$, and on $\ss_\oo$.

\begin{prop}
$\ss_\oo$ is the orbit of ${\bf e_1}$  under the action of $\u_2(\a)$.
\end{prop}
\begin{proof}
Clearly the orbit $\u_2(\a)\cdot{\bf e_1}$ is contained in  $\ss_\oo$. Pick $\xx=\tilde{g}{\bf e_1} \in\ss_\oo$. We must show that there exists a unitary element $\tilde{u}$ such that $\tilde{u}{\bf e_1}=\tilde{g}{\bf e_1}$.
Denote  $\tilde{r}=\tilde{g}^*\tilde{g}$ and by $\tilde{p}_0$ the selfadjoint projection $\left( \begin{array}{cc} 1 & 0 \\ 0 & 0 \end{array} \right)$.  The fact that $\tilde{g}{\bf e_1}\in\ss_\oo$ means that
$$
\langle \tilde{r} {\bf e_1}, {\bf e_1}\rangle =1,
$$
or, equivalently, that $\tilde{r}\tilde{p}_0=\tilde{p}_0$. It follows that, for any $n\ge 0$, $\tilde{r}^n\tilde{p}_0=\tilde{p}_0$, and thus, if $q(t)=\sum_{i=0}^n \alpha_i t^i$, then $q(\tilde{r})\tilde{p}_0=\sum_{i=1}^n \alpha_i \tilde{p}_0=q(1)\tilde{p}_0$. Let $q_n(t)$ be a sequence of polynomials which converge uniformly to $\sqrt{t}$ in the spectrum of $\tilde{r}$ (a compact subset of $(0,+\infty)$). Then $q_n(\tilde{r})\to \tilde{r}^{1/2}$. On the other hand, it holds
$$
q_n(\tilde{r})\tilde{p}_0=q_n(1)\tilde{p}_0\to \tilde{p}_0.
$$
Then, $\tilde{r}^{1/2}\tilde{p}_0=\tilde{p}_0$. Consider the polar decomposition of the (invertible) element $\tilde{g}$ in $M_2(\a)$: $\tilde{g}=\tilde{u} \tilde{r}^{1/2}$. Then $\tilde{u}$ is a unitary element such that
$\tilde{u}\tilde{p}_0=\tilde{u}\tilde{r}\tilde{p}_0=\tilde{g}\tilde{p}_0$. In particular,
$$
\tilde{u}{\bf e_1}=\hbox{ first column of } \tilde{u}\tilde{p}=\hbox{ first column of } \tilde{g}\tilde{p}=\tilde{g}{\bf e_1}.
$$ 
\end{proof}
\begin{coro}
The unitary group $\u$ of $\a$ acts on $\ss_\oo$ by right multiplication.
\end{coro}
\begin{proof}
Pick $\xx=\tilde{u}\cdot {\bf e_1}\in \ss_\oo$ ($\tilde{u}\in \u_2(\a)$), and $v\in\u$. Then
$$
\xx\cdot v=\tilde{u}\cdot {\bf e_1} \cdot v=\tilde{u}v \cdot {\bf e_1},
$$
where $\tilde{u}v=\tilde{u}\left( \begin{array}{cc} v & 0 \\ 0 & v \end{array} \right)\in\u_2(\a)$.
\end{proof}

\begin{rem}
There is a retraction $\a^2_{reg}\to \ss$: $\xx\mapsto \xx \langle \xx, \xx \rangle ^{-1/2}$. 
\end{rem}

\begin{coro}\label{conexa}
In the case $\a=\b(\h)$,  $\ss_\oo=\ss$ and $\oo=\a^2_{reg}$.
\end{coro}

\section{The projective line over $\a$}

We shall call {\it rank one submodule} a closed submodule $\ell\subset \a^2$ which is orthocomplemented,  i.e., 
$$
\ell\oplus\ell^\perp=\a^2 ,\ \hbox{ where } \ell^\perp=\{\yy\in\a^2: <\xx,\yy>=0 \hbox{ for all } \xx\in\ell\},
$$
and such that there exists $\xx_0\in\ell$ with $[x_0]=\{x_0a: a \in \a\}=\ell$.
As the result below shows, not every element in $\a^2$ generates a rank one submodule:
\begin{prop}
Let $\xx\in\a^2$, $\xx\ne 0$. Then $\ell=\{\xx a: a\in\a\}$ is a rank one submodule if and only if there exists $a_0\in\a$ such that $\xx_0=\xx a$ is also a generator for $\ell$, and $<\xx_0,\xx_0>$ is a projection in $\a$.
\end{prop}
\begin{proof}
Suppose that there is a generator $\xx_0$ of $\ell$ such that $<\xx_0,\xx_0>=p_0$ is a projection. Note that $\xx_0=\xx_0p_0$. Indeed, $\xx_0(1-p_0)$ satisfies $<\xx_0(1-p_0),\xx_0(1-p_0)>=(1-p_0)p_0(1-p_0)=0$, i.e., $\xx_0(1-p_0)=0$. Then $p_{\xx_0}$ defined by $p_{\xx_0}(\yy)=\xx_0<\xx_0,\yy>$ is the orthogonal projection onto $\ell$. Clearly, it is selfadjoint. It is a projection: 
$$
p_{\xx_0}^2(\yy)=\xx_0<\xx_0,\xx_0<\xx_0,\yy>>=\xx_0 p_0<\xx_0,\yy>=\xx_0<\xx_0,\yy>=p_{\xx_0}(\yy).
$$
Clearly, the range of $p_{\xx_0}$ is $\ell$.

Conversely, suppose that $\ell=[\xx]$ is orthocomplemented. Then there exists a symmetric adjointable projection $p:\a^2\to\ell\subset \a^2$. It is of the form $p(\yy)=\xx \varphi(\yy)$, where $\varphi:\a^2\to\a$ is $\a$-linear and bounded. Then (see \cite{lance}, p. 13) $\varphi$  is of the form $\varphi(\yy)=<\zz,\yy>$ for some $\zz\in\a^2$. The adjoint of $p(\yy)=\xx<\zz,\yy>$ is $p^*(\yy)=\zz<\xx,\yy>$. Therefore, the fact that $p^*=p$,  implies that $\zz=\xx a$ for some $a\in\a$. Thus, $p=\xx a<\xx,\ \  >=\xx<\xx a, \ \  >=xa^*<\xx, \ \ >$. Using that $p^2=p$,
we get that 
$$
p=p^2=\xx a^*<\xx,\xx a<\xx,  \ \ >>=\xx a^*<\xx,\xx>a<\xx, \ \ >=\xx b<\xx b,\  \ >, 
$$
where $b=(a^*<\xx,\xx>a)^{1/2}$. Denote $\xx_0=\xx b$. The above computation shows that $\xx_0$ is also a generator for $\ell$. 
Then $p=p_{\xx_0}$. Using again the fact that $p^2=p$, we obtain
$$
\xx_0<\xx_0,\ \  >=p=p^2=\xx_0<\xx_0,\xx_0><\xx_0, \ \ >.
$$
Evaluating at  $\xx_0$, we get $\xx_0<\xx_0,\xx_0>=\xx_0<\xx_0,\xx_0>^2$; applying $<\xx_0,\ \ >$ to this equality, we obtain that $c=<\xx_0,\xx_0>$ satisfies $c^2=c^3$. Since $c\ge 0$, this means that $c$ has two spectral values $0,1$, and thus $c$ is a (non zero) projection.
\end{proof}

In this paper, we shall be interested in rank one modules $\ell=[\xx]$ with $\xx\in\oo$. Note that for these special generators, the projection $p_0$ in the above proposition is $<\xx,\xx>=1$.

Define
$$
\pa=\{\ell=[\xx]: \xx\in\oo\}.
$$

An equivalence relation is defined in $\oo$: $\xx\sim\xx'$ iff  there exists  $a\in G$  such that $\xx'=\xx a$ (or equivalently, $[\xx]=[\xx']$). Then, it holds that also $\pa=\oo/\sim$.

The map
$$
\oo\to \pa \, \ \ \xx\mapsto [\xx]
$$
is onto, and the fibers  are the equivalence classes  in $\oo$.

Consider the map 
$$
\oo\to \ss_\oo \ , \  \ \ \xx\mapsto \xx \langle\xx,\xx\rangle^{-1/2}.
$$
Clearly, it is  a well defined and continuous retraction.

\begin{rem}
It should be noticed that the projective spaces of C$^*$-modules over $\a$ which appear in \cite{paisanos1} and \cite{paisanos2} are strictly bigger than $\pa$.
\end{rem}

We end this section by defining the {\it cross ratio}  $CR(\ell_1, \ell_2, \ell_3, \ell_4)$ of four submodules $\ell_1, \ell_2, \ell_3, \ell_4$ in $\pa$, following ideas by M.I. Zelikin \cite{zelikin}. 

\begin{defi}
We denote by $CR(\ell_1, \ell_2, \ell_3, \ell_4)$  the (possibly empty) set of module homomorphisms $\varphi: \ell_3\to \ell_3$ of the form $\varphi=\psi \eta$, where the homomorphisms $\eta:\ell_3\to\ell_2$, $\psi:\ell_2\to\ell_3$ satisfy $x-\psi(x)\in\ell_4$ and $y-\eta(y)\in\ell_1$, for all $x\in\ell_3$, $y\in\ell_2$.
\end{defi}

\section{The hyperbolic part of the $\a$-projective line:  the Poincar\'e disk of $\a$.}

In this section we collect from \cite{tejemas} certain facts concerning the hyperbolic geometry of the unit disk $\d$ of $\a$. Recall that, with the appropriate metric, it is a model for the part of $\pa$ which carries a nonpositively curved geometry. The main feature is a signed sesquilinear form $\theta$, and the group which preserves it.
More precisely, define
$$
\theta:\a^2\times \a^2 \to \a \ ,\ \theta(\xx,\yy)=x_1^*y_1-x_2^*y_2.
$$
In \cite{tejemas} it was denoted $\theta_D$.
Note that $\theta(\xx,\yy)=\langle \vecx, \rho_\theta \vecy \rangle$, where 
$$
\rho=\rho_\theta=\left(\begin{array}{cc} 1 & 0 \\ 0 & -1 \end{array} \right).
$$ 
The main tool in the study of the geometry of the Poincar\'e disk of $\a$ is the subgroup $\u(\theta)$ of $Gl_2(\a)$ which leaves invariant the  form $\theta$:
$$
\u(\theta)=\{\tilde{g} \in Gl_2(\a): \theta (\tilde{g} \xx, \tilde{g} \yy)=\theta(\xx,\yy) \hbox{ for all } \xx,\yy\in\a^2\}.
$$
We call $\u(\theta)$ the {\it unitary group of the form} $\theta$.

Equivalently, $\tilde{g}\in Gl_2(\a)$ belongs to $\u(\theta)$ iff $\tilde{g}^*\rho \tilde{g}=\rho$.
In \cite{tejemas} it was shown that $\u(\theta)$ is a  complemented Banach-Lie subgroup of $Gl_2(\a)$.

Let us define the {\it hyperbolic part} $\pa^\theta$ of $\pa$.
\begin{defi}
The set
$$
\pa^\theta=\{[\xx]\in\pa: \hbox{there exists a generator } \xx_0 \hbox{ of  } [\xx] \hbox{ such that  } \theta(\xx_0,\xx_0)\in G^+\}
$$
is called the hyperbolic part of $\pa$.
\end{defi}
\begin{rem}
Note that 
$$
\theta(\xx a, \xx a)=a^*x_1^*x_1a-a^*x_2^*x_2a=a^*(x_1^*x_2-x_2^*x_2)a=a^*\theta(\xx,\xx) a.
$$
Therefore, if $\theta(\xx_0,\xx_0)\in G^+$ for a generator $\xx_0$ of $[\xx]$, then the same is true for any other generator of $[\xx]$. 

Also, it holds that $\theta$ is positive (semidefinite) in $[\xx]$. But the converse does not hold: $\theta$  might be positive in $[\xx]$  without $\theta(\xx_0,\xx_0)$ being invertible for any generator $\xx_0$.
\end{rem}

\begin{teo}\label{transitiva}
The group $\u(\theta)$ acts transitively in $\pa^\theta$. If $\xx\in\oo$ and $\tilde{g}\in\u(\theta)$, the action is given by
$$
\tilde{g}\cdot[\xx]=[\tilde{g}\xx].
$$
\end{teo}
Before we give the proof of this statement, we introduce certain facts on $\pa^\theta$.
Note that if $\xx\in\oo$ is a generator of $[x]\in\pa^\theta$, then 
$$
\theta(\tilde{g}\xx,\tilde{g}\xx)=\theta(\xx,\xx)\in G^+.
$$
Also it is clear that $\tilde{g}\xx\in\oo$. Thus,  $\u(\theta)$ acts in $\pa^\theta$. To prove the above theorem, we only need to  show  that  the action is transitive, i.e., that if $\xx$ satisfies $\theta(\xx,\xx)\in G^+$ (i.e. $\xx\in\oo$), then there exists $\tilde{h}\in\u(\theta)$ such that $\xx=\tilde{h}{\bf e_1}$.

\begin{defi}
Let us denote by $\k_\theta$ the hyperboloid defined by the form $\theta$:
$$
\k_\theta=\{\xx\in\a^2: \theta(\xx,\xx)=1 \hbox{ and } x_1\in G\}.
$$
\end{defi}
The space $\k_\theta$ will be a sort of coordinate space for $\pa^\theta$.   Note that $\k_\theta\subset\a^2_{reg}$: $\theta(\xx,\xx)=1$ means that $x_1^*x_1=1+x_2^*x_2$, so that $\langle \xx,\xx\rangle =1+2 x_2^*x_2\in G^+$. Moreover, we shall see below that $\k_\theta\subset \oo$.
Let us recall the following facts from \cite{tejemas} (Sections 4 and 5):

\begin{prop}\label{tejemas}
With the current notations, we have the following:
\begin{enumerate}
\item
The group $\u(\theta)$ acts  on $\k_\theta$: if  $\xx\in\k_\theta$ and $\tilde{g}\in\u(\theta)$, then $\tilde{g} \xx\in\k_\theta$. The action is transitive. In particular, $\k_\theta=\{\tilde{g} {\bf e_1}: \tilde{g}\in\u(\theta)\} \subset \oo$.
\item
There is a (complemented Banach-Lie) subgroup $\b_\theta$ of $\u(\theta)$ which acts freely and transitively in $\k_\theta$. In particular, $\k_\theta$ has  group structure.
\item
$\k_\theta$ is a complemented $C^\infty$-submanifold of $\a^2$. The map
$$
\pi_\theta:\k_\theta \to \pa^\theta \ , \ \ \pi_\theta(\xx)=[\xx]
$$
is a $C^\infty$ submersion. In particular, it is onto. If $\xx,\yy\in\k_\theta$ satisfy $[\xx]=[\yy]$, then there exists $u\in\u$ such that $\yy=\xx u$, i.e., the fibers of $\pi_\theta$ identify with $\u$.
\item The map $\pi_\theta$ is $\u(\theta)$-equivariant: if $\tilde{g}\in\u(\theta)$ and $\xx\in\k_\theta$,
$$
\pi_\theta(\tilde{g} \xx)=[\tilde{g}\xx]=\tilde{g}\cdot[\xx]=\tilde{g}\cdot \pi_\theta(\xx). 
$$
\end{enumerate}
\begin{proof} (of Theorem \ref{transitiva})
Note that ${\bf e_1}\in\k_\theta$. Therefore if $\xx\in\k_\theta$, there exists $\tilde{g}\in\u(\theta)$ such that $\tilde{g}{\bf e_1}=\xx$. Since $\pi_\theta$ is onto, it follows that any $[\xx]\in \pa^\theta$ has a generator (say) $\xx\in\k_\theta$. Then $[\xx]=\tilde{g}\cdot[{\bf e_1}]$.
\end{proof}
\end{prop}
\begin{rem}\label{tejemasbis}

 Let us describe the group $\b_\theta$ explicitly:
$$
\b_\theta= \{ \left(  \displaystyle{\begin{array}{cc} \displaystyle{\frac{g+\hat{g}}{2}} - \hat{g}x & \displaystyle{\frac{g-\hat{g}}{2}}-\hat{g}x \\ \displaystyle{\frac{g-\hat{g}}{2}}+\hat{g}x & \displaystyle{\frac{g+\hat{g}}{2}} +\hat{g}x \end{array}} \right): g\in G ,  x^*=-x \} .
$$
In fact, the group used in \cite{tejemas} is $\b=\{\left(\begin{array}{cc} g & 0 \\ \hat{g}s & \hat{g} \end{array}\right): g\in G, s^*=s\}$.  The group considered here is conjugate of $\b$: 
$$
\b_\theta=U^* \b U,
$$
where $U=\frac1{\sqrt{2}}\left( \begin{array}{cc} 1 & 1 \\ i & -i \end{array} \right)$.
Indeed, if $g\in G$, and $\tilde{g}=\left(\begin{array}{cc} g & 0 \\ \hat{g}s & \hat{g} \end{array}\right)\in\b$, then
$$
U^*\tilde{g}U=\frac12\left(\begin{array}{cc}g+\hat{g} -i \hat{g} s & g-\hat{g} -i\hat{g}s \\ g-\hat{g} +i \hat{g}s & g+\hat{g}+i\hat{g}s \end{array}\right),
$$
where $x=\frac{i}{2}  s$ is anti-Hermitian.We call $\b_\theta$ the {\it Borel subgroup} of $\u(\theta)$. The corresponding Borel subgroup of $\u(\theta_H)$ for the Poincar\'e halfspace $\h$ was described in \cite{tejemas}. 
 The facts mentioned in Proposition \ref{tejemas} where proved in \cite{tejemas} for the hyperboloid $\k_{\theta_H}=U\k_\theta$, of the form  $\theta_H(\xx,\yy)=\theta(U^*\xx,U^*\yy)$).
\end{rem} 

\begin{rem}
There are two natural subgroups in $\b_\theta$, which are homomorphic images of the invertible group $G$ of $\a$ and the additive group $\a_{ah}$ of anti-Hermitian elements of $\a$, which we denote, respectively,  by $G_\theta$ and $T_\theta$. Namely,
$$
G_\theta=\{\tilde{\gamma}_g:=\frac12 \left(\begin{array}{cc} g+\hat{g} & g-\hat{g} \\ g-\hat{g} & g+\hat{g} \end{array}\right): g\in G\} \ \hbox{ and } \ 
T_\theta=\{\tilde{\tau}_x:=\left(\begin{array}{cc} 1-x & -x \\ x & 1+x \end{array}\right): x^*=-x\}.
$$
Elementary computations show that
\begin{enumerate}
\item
$G_\theta\subset \b_\theta$ and $T_\theta\subset\b_\theta$.
\item
The maps $G\ni g\mapsto \tilde{\gamma}_g\in G_\theta$ and $\a_{ah}\ni x\mapsto \tilde{\tau}_x\in T_\theta$ are group homomorphisms ($\a_{ah}$ considered with its additive structure). In particular,  $G_\theta$ and $T_\theta$ are subgroups of $\b_\theta$.
\item
$G_\theta$ and $T_\theta$  generate $\b_\theta$. More precisely, if $\tilde{g}\in\b_\theta$, $\tilde{g}=\left(  \displaystyle{\begin{array}{cc} \displaystyle{\frac{g+\hat{g}}{2}} - \hat{g}x & \displaystyle{\frac{g-\hat{g}}{2}}-\hat{g}x \\ \displaystyle{\frac{g-\hat{g}}{2}}+\hat{g}x & \displaystyle{\frac{g+\hat{g}}{2}} +\hat{g}x \end{array}} \right)$, then
$$
\tilde{g}=\frac12 \left(\begin{array}{cc} g+\hat{g} & g-\hat{g} \\ g-\hat{g} & g+\hat{g} \end{array}\right)\left(\begin{array}{cc} 1-x & -x \\ x & 1+x \end{array}\right).
$$
\item
The above factorization is unique. Or, equivalently, $G_\theta\cap T_\theta$ contains only the identity matrix.

\end{enumerate}
\end{rem}

Let us describe an alternative model for $\pa^\theta$, namely, the unit disk  $\d=\{a\in\a: \|a\|<1\}$ of $\a$.
There is a natural map from $\k_\theta$ to $\d$:
\begin{lem}\label{lema48}
The map $\k_\theta\to \d$, $\xx\mapsto x_2x_1^{-1}$ is well defined, onto and $C^\infty$.
\end{lem}
\begin{proof}
Since $\theta(\xx,\xx)=1$, $x_2^*x_2=x_1^*x_1-1$. Then
$$
0\le (x_2x_1^{-1})^*x_2x_1^{-1}=(x_1^{-1})^*(x_1^*x_1-1)x_1^{-1}=1-(x_1x_1^*)^{-1}.
$$
Then 
$$
\|x_2x_1^{-1}\|^2=\|(x_2x_1^{-1})^*x_2x_1^{-1}\|=\|1-(x_1x_1^*)^{-1}\|=\sup\{ 1-\frac1{t}: t\in\sigma(x_1x_1^*)\}<1,
$$
because $\sigma(x_1x_1^*)$ is a compact set in $(0,+\infty)$ (recall that $x_1$ is invertible).
The map is clearly $C^\infty$. Pick $a\in\d$. Then $1-a^*a\in G$. Put $x_1=(1-a^*a)^{-1/2}$ and $x_2=ax_1$. Then, clearly, $\xx=\vecx$ belongs to $\k_\theta$ and is mapped to $a$.
\end{proof}
\begin{prop}\label{prop49}
The map $\k_\theta\to \d$ induces a $C^\infty$ diffeomorphism 
$$
\pa^\theta\stackrel{\simeq}{\longrightarrow} \d  \ , \ \ [\xx]\mapsto x_2x_1^{-1}.
$$  
\end{prop}
\begin{proof}
If $\xx,\yy\in\oo$ satisfy $[\xx]=[\yy]$, then there exists $g\in G$ such that $\yy=\xx g$, and thus $y_2y_1^{-1}=x_2g (x_1g)^{-1}=x_2x_1^{-1}$. Its inverse is 
$$
\d\ni z\mapsto \left[ \left( \begin{array}{c} 1 \\ z \end{array}  \right)\right] \in \pa^\theta.
$$
\end{proof}
The map $\pi_\theta:\k_\theta\to\pa^\theta$, $\pi_\theta(\xx)=[\xx]$  (or alternatively, $\tilde{\pi}_\theta:\k_\theta\to\d$, $\tilde{\pi}_\theta(\xx)=x_2x_1^{-1}$) is an analogue of the classical Hopf fibration.

It will be useful to describe the action of $\u(\theta)$ on the model $\d$. By straightforward computations, if $\tilde{g}\in\u(\theta)$ and $z\in\d$, then

\begin{equation}\label{accion en D}
\tilde{g}\cdot z=(g_{21}+g_{22}z)(g_{11}+g_{12}z)^{-1}.
\end{equation}
\begin{rem} \label{isotropiaD}
In particular, note that if $\tilde{k}\in\u(\theta)$ satisfies $\tilde{k}\cdot 0=0$, then
$\tilde{k}=\left( \begin{array}{cc} u_1 & 0 \\ 0 & u_2 \end{array} \right)$, with $u_1,u_2\in\u_\a$.
\end{rem}

We finish this section by addressing the characterization of the local regular structure of $\pa^\theta$. Instead of exhibiting an atlas of local coordinates, we refer to an  intrinsic model for $\pa^\theta$. In \cite{tejemas} we studied the space $\q_\rho$, which is defined as  the space 
$$
\q_{\rho}=\{\epsilon\in M_2(\a): \epsilon^2=1  \hbox{ and } \rho\epsilon\in G^+\}.
$$ 
The elements $\epsilon\in M_2(\a)$ satisfying $\epsilon^2=1$, usually called reflections, are in (natural) one to one correspondence with projections $q\in M_2(\a)$, $q^2=q$: $\epsilon\leftrightarrow q=\frac12(1+\epsilon)$. The condition $\epsilon\rho\in G^+$ implies, in particular, that $\epsilon$, and thus $q$, is $\theta$-selfadjoint. More precisely, the projections $q$ in $\q_\rho$, correspond to the submodules $\ell\in\pa^\theta$ (see \cite{tejemas}), by means of the one to one mapping
$$
\q_\rho \ni q \longleftrightarrow \ell=R(q)\in\pa^\theta.
$$
In \cite{prILLINOIS} it was shown that $\q_\rho$ is a complemented submanifold of $M_2(\a)$. A benefit we obtain from this coordinate free identification $\pa^\theta\simeq\q_\rho$, is a description of the tangent spaces of $\pa^\theta$ as subspaces of $\theta$-selfadjoint elements in $M_2(\a)$. We shall specify this in the next section.

\section{The tangent spaces of $\pa^\theta$.}\label{seccion corta}

In this short section we characterize the tangent spaces of the  projective line $\pa$, and its hyperbolic part $\pa^\theta$. If 
$\xx_0\in\oo$, we identify $\pa$  with $\oo / \xx_0 \cdot G$,  because $\oo$ is open in $\a^2$ (and $G$ is open in $\a$). Then  we have
$$
(T\pa)_{[\xx_0]} \simeq \a^2\  / \ \xx_0\cdot \a=\a^2 /[\xx_0].
$$

Note that $\pa^\theta$ is open in $\pa$.

Therefore, if $[\xx_0]\in \pa^\theta$, 
$$
(T\pa^\theta)_{[\xx_0]}= (T\pa)_{[\xx_0]} \simeq \a^2\  / \ \xx_0\cdot \a=\a^2 /[\xx_0].
$$
If $[\xx]\in\pa$, let us denote by $[\xx]^{\perp_\theta}$ the $\theta$-orthogonal complement of $[\xx]$.

\begin{lem}
If $\xx_0\in\oo$, then there exists the submodule $[\xx_0]^{\perp_\theta}$. It is generated by an element $\yy_0\in\a^2_{reg}$ (not necessarily in $\oo$).
\end{lem}
\begin{proof}
Consider
$$
\yy_0=\left( \begin{array}{c} (x_1^*)^{-1}x_2^* \\ 1 \end{array} \right).
$$
Straightforward computations show that $\theta(\xx_0,\yy_0)=0$. It is easy to see that $\xx_0$ and $\yy_0$ generate $\a^2$.
\end{proof}
\begin{lem}\label{forma negativa}
If $[\xx_0]\in\pa^\theta$, then the  form $\theta$ is negative and non degenerate in $[\xx_0]^{\perp_\theta}$: if $\yy\in [\xx_0]^{\perp_\theta}$, then $-\theta(\yy,\yy)\in G^+$.
\end{lem}
\begin{proof}
Recall that $[\xx_0]\in\pa^\theta$ means that for any generator (e.g. $\xx_0=\vecx$), it holds that $x_1\in G$ and $\theta(\xx_0,\xx_0)\in G^+$. Put $\yy_0=\left( \begin{array}{c} (x_1^*)^{-1}x_2^* \\ 1 \end{array} \right)$ as above.
Then, for any $\yy=\yy_0 a\in[\xx_0]^\perp$,
$$
\theta(\yy,\yy)=a^*\theta(\yy_0,\yy_0) a,
$$
and 
$$
\theta(\yy_0,\yy_0)=x_2x_1^{-1}(x_1^*)^{-1}x_2^*-1=x_2x_1^{-1}(x_2x_1^{-1})^*-1.
$$
From Lemma \ref{lema48} and Proposition \ref{prop49}, we get that $x_2x_1^{-1}\in\d$, i.e., $\|x_2x_1^{-1}\|<1$; thus, $1-x_2x_1^{-1}(x_2x_1^{-1})^*\in G^+$.
\end{proof}

For $[\xx_0]\in\pa^\theta$, the tangent space $(T\pa^\theta)_{[\xx_0]}=\a^2 /[\xx_0]$ is isomorphic to any supplement of $[\xx]$ in $\a^2$. We choose to identify 
\begin{equation}\label{tangente}
(T\pa^\theta)_{[\xx_0]}\simeq [\xx_0]^{\perp_\theta}.
\end{equation}
\begin{rem}\label{como cambia}
A remark is in order. If $\ell=[\xx_0]\in\pa^\theta$, the identification of $(T\pa^\theta)_\ell$  with $\ell^{\perp_\theta}$ depends on the choice of the generator $\xx_0$. In order to see how the change of generator affects this identification, we must refer to an intrinsic model for $\pa^\theta$. We  choose the model $\q_\rho$ described in the first section. Also, to keep matters more simple, consider generators in $\k_\theta$, i.e. $\xx_0$ satisfies $\theta(\xx_0,\xx_0)=1$. If $\ell=q_\ell$ (the unique $\theta$-orthogonal projection onto $\ell$) then 
$$
(T\pa^\theta)_\ell=(T\q_\rho)_{q_\ell}=\{X\in M_2(\a): X \hbox{ is } \theta-\hbox{symmetric and } \rho-\hbox{co-diagonal}\}.
$$
 If $\xx'_0\in\k_\theta$ is another generator of $\ell$, then there exists a unitary $u\in\u_\a$ such that $\xx_0'=\xx_0u$. 
In the identification between $(T\pa^\theta)_\ell$ and $\a^2/\ell$ done by means of $\xx_0$, the tangent vector $X$ is identified with  $X\xx_0\in\ell^{\perp_\theta}$.

Thus, if we change to $\xx'_0$, both identifications differ in right multiplication by $u$.  
\end{rem}
\section{The geometry of  the disk}
 
In the previous section we introduced an $\u(\theta)$-invariant Finsler metric in the hyperbolic part of the $\a$-projective line $\pa^\theta$ induced by the quadratic form $\theta$.
Also, we noted that there is a natural diffeomorphism $[\xx]\mapsto x_2x_2^{-1}$ between $\pa^\theta$ and the unit disk $\d$ of $\a$

In \cite{tejemas}, we introduced a (nonpositively curved) metric in $\d$, by establishing in turn an identification between $\d$ and a space of positive operators related to the symmetry $\rho$ (related to $\theta$).  Let us briefly describe it:
\begin{rem}\label{tejemasbisbis}(several results taken from \cite{tejemas})
\begin{enumerate}
\item
$\d$ is embedded in the space $Gl_2(\a)^+$  of positive elements in $Gl_2(\a)$ by means of the map
$$
\Phi_\d:\d \to GL_2(\a)^+ \ , \  \ \Phi_\d(a)=\left( \begin{array}{cc} 2(1-a^*a)^{-1}-1 & -2(1-a^*a)^{-1}a^*
\\ -2a(1-a^*a)^{-1} & 2a(1-a^*a)^{-1}a^* +1 \end{array} \right)
$$
$$
=-\rho+2 \left( \begin{array}{cc} (1-a^*a)^{-1} & 0 \\ 0 & (1-aa^*)^{-1} \end{array} \right)\left( \begin{array}{cc} 1  & -a^* \\ -a & aa^* \end{array} \right).
$$
For the last equality, we use that $a(1-a^*a)^{-1}=(1-aa^*)^{-1}a$.
Also note that
$$
\Phi_\d(a)=\left( \begin{array}{cc} (1-a^*a)^{-1} & 0 \\ 0 & (1-aa^*)^{-1} \end{array} \right)\left( \begin{array}{cc} 1+a^*a  & -2a^* \\ -2a & 1+aa^* \end{array} \right)
$$
where both matrices commute.

\item
Therefore, given two points $z_0 ,z_1\in\d$ there exists a unique geodesic joining them. In \cite{tejemas} we computed the velocity of this unique geodesic in the case $z_0=0$ and $z_1=z$ is an arbitrary element of $\d$. The geodesic is given by
$$
\delta(t)=e^{t\left(\begin{array}{cc} 0 & \alpha^* \\ \alpha & 0 \end{array}\right)} \cdot 0,
$$ 
where
\begin{equation}\label{formula alfa}
\alpha=z\sum_{k=0}^\infty \frac{1}{2k+1} (z^*z)^k,
\end{equation}
satisfies that $\delta(0)=0$ and $\delta(1)=z$. Note here that the series $\sum_{k=0}^\infty \frac{t^{2k+1}}{2k+1}$ corresponds, in the interval $(-1,1)$, to the function $f(t)=\frac12\log(\frac{1+t}{1-t})$. We shall compute below the explicit form of $\delta$ in $\d$.
\item
The norm of $(T\d)_0$ is the usual norm of $\a$.

\item
Recall from Proposition \ref{tejemas} the hyperboloid $\k_\theta$. We have the following conmutative diagram:
\begin{equation}\label{diagrama}
\xymatrix{
& \k_\theta \ar[ld]_{\hat{\pi_\theta }}\ar[rd]^{\tilde{\pi}_\theta} \\
\pa^\theta \ar[rr]^{\simeq}
&& \d ,
}
\end{equation}
where the map $\k_\theta\to\pa^\theta$ is the restriction of the quotient map ($\oo\to\pa$). The group $\u(\theta)$ acts on the three spaces, and the maps are equivariant with respect to the action.
\end{enumerate}
\end{rem}
Let us compute explicitly the form of the geodesic $\delta$ joining $0$ and $z\in\d$ at time $t=1$:
\begin{lem}\label{lema62}
Given $z\in\d$, the unique geodesic $\delta$ of $\d$ with $\delta(0)=0$ and $\delta(1)=z$ is given by
\begin{equation}\label{delta}
\delta(t)= \omega \tanh(t|\alpha|),
\end{equation}
where  $\alpha$ is given in {\rm (\ref{formula alfa})} above, and $\omega$ is the partial isometry in the polar decomposition of $\alpha$: $\alpha=\omega |\alpha|$, performed in $\a^{* *}$. 
\end{lem}
\begin{proof}
Straightforward computations show that the even and odd powers of $\left(\begin{array}{cc} 0 & \alpha^* \\ \alpha & 0 \end{array}\right)$ are, respectively
$$
\left(\begin{array}{cc} 0 & \alpha^* \\ \alpha & 0 \end{array}\right)^{2k}=\left(\begin{array}{cc} (\alpha^*\alpha)^k & 0 \\ 0 & (\alpha\alpha^*)^k \end{array}\right) \hbox{ and } \left(\begin{array}{cc} 0 & \alpha^* \\ \alpha & 0 \end{array}\right)^{2k+1}= \left(\begin{array}{cc} 0 & \alpha^* (\alpha\alpha^*)^k\\ \alpha(\alpha^*\alpha)^k & 0 \end{array}\right).
$$
We are interested in the first column of $e^{t\left(\begin{array}{cc} 0 & \alpha^* \\ \alpha & 0 \end{array}\right)}$, which is
$$
\left( \displaystyle{ \begin{array}{l} \sum_{k=0}^\infty \frac{t^{2k}}{(2k)!} (\alpha^*\alpha)^k \\ \alpha\sum_{k=0}^\infty \frac{t^{2k+1}}{(2k+1)!} (\alpha^*\alpha)^k\end{array}}\right)= \left( \begin{array}{l} \cosh(t|\alpha|) \\ \omega \sinh(t|\alpha|) \end{array}\right).
$$
Then 
$$
\delta(t)=\omega \sinh(t|\alpha|)(\cosh(t|\alpha|))^{-1}=\omega \tanh(t|\alpha|).
$$
\end{proof}
Notice that $\omega\in\a^{**}$ need not belong to $\a$. However $\delta(t)\in\a$ for all $t$.

Let us relate the polar decompositions of $z$ and $\alpha$.
\begin{prop}
If $z\in\d$ and $\alpha$ as in {\rm (\ref{formula alfa})}, then
$$
\alpha=\frac12\omega (\log(1+|z|)-\log(1-|z|))  ,  \ \ \hbox{ and } \  \  \  z=\omega |z|,
$$
i.e., the partial isometry $\omega\in\a^{**}$ is the same for $\alpha$ and $z$.
\end{prop}
\begin{proof}
First note that 
$$
|\alpha|^2=\alpha^*\alpha=\sum_{k=0}^\infty \frac{1}{2k+1} (z^*z)^k z^*z\sum_{k=0}^\infty \frac{1}{2k+1} (z^*z)^k=\left(\sum_{k=0}^\infty \frac{1}{2k+1} |z|^{2k+1}\right)^2,
$$
i.e., $|\alpha|=\frac12\left(\log(1+|z|) - \log(1-|z|)\right)$.  Next, put $z=\mu|z|$ the polar decomposition of $z$. Note that,  since $\alpha=z\sum_{k=0}^\infty \frac{1}{2k+1} |z|^{2k}$, we have that
$$
\alpha=\mu |z| \sum_{k=0}^\infty \frac{1}{2k+1} |z|^{2k}=\mu |\alpha|.
$$
Thus, in order to prove our claim, it suffices to show that both partial isometries $\mu, \omega$ have the same initial and final spaces (the result follows, then, by the uniqueness property of the polar decomposition).
The partial isometry $\mu$ maps $N(z)^\perp=N(|z|)^\perp$  onto $\overline{R(z)}$, whereas $\omega$ maps $N(\alpha)^\perp$ onto $\overline{R(\alpha)}$. If $z\xi=0$, 
then
$$
\alpha \xi=z\sum_{k=0}^\infty \frac{1}{2k+1} (z^*z)^k \xi= \sum_{k=0}^\infty \frac{1}{2k+1} (zz^*)^kz\xi=0,
$$
i.e., $N(z)\subset N(\alpha)$. Conversely, in the last expression of $\alpha$,
$\alpha= \sum_{k=0}^\infty \frac{1}{2k+1} |z^*|^{2k}z$; observe that 
$$
\sum_{k=0}^\infty \frac{1}{2k+1} |z^*|^{2k}=g(|z^*|),
$$
where $g(t)=\frac{1}{2t}(\log(1+t)-\log(1-t))$ (which can be extended continuously as $g(0)=1$), is defined in $\sigma(|z^*|)\subset [0,1)$, and is nonvanishing there. Therefore, $g(|z^*|)$ is invertible. Thus, $\alpha\xi=0$ implies $z\xi=0$.  Again, using the function $g$, we get $\alpha=z g(|z|)$, and thus $R(z)=R(\alpha)$.
\end{proof}
\begin{coro}\label{64}
The exponential map $\Exp_0$ of $\d$ at $0$, and its inverse $\Log_0$ can be written explicitly  as follows:  if $z=\omega |z|\in\d$
$$
\Log_0:\d\to (T\d)_0 \ , \ \Log_0(z) \frac12 \omega \log\left( (1+|z|)(1-|z|)^{-1}\right).
$$
If $\alpha=\omega|\alpha|\in\a\simeq(T\d)_0$, then
$$
\Exp_0:(T\d)_0\to \d \ , \ \Exp_0(\alpha)=\omega \tanh(|\alpha|).
$$
In particular, if $\alpha=\Log_0(z)$ (or, equivalently, $z=\Exp_0(\alpha)$), then $z$ and $\alpha$ have the same partial isometry in the polar decomposition. Also,
$$
|\Exp_0(\alpha)|=\tanh(|\alpha|)  \ \hbox{ and } \ |\Log_0(z)|=\frac12\log\left( (1+|z|)(1-|z|)^{-1}\right).
$$
\end{coro}
\section{Limit points of geodesics}
One of our  concerns in computing the geodesic $\delta$, and the above results on the polar decomposition, is to establish the following result:
\begin{teo}
For $z\in\d$, let $\delta$ be the unique geodesic of $\d$ such that $\delta(0)=0$ and $\delta(1)=z$. Put $z=\omega |z|$ the polar decomposition (i.e., $\omega\in\a^{**}$); then
$$
{\rm SOT}-\lim_{t\to\infty} \delta(t)=\omega \ \ \hbox{ and } \ \ {\rm SOT}-\lim_{t\to-\infty} \delta(t)=-\omega .
$$
\end{teo}
\begin{proof}
By formula (\ref{delta}), we only need to compute the limit of $\tanh(t|z|)$ when $t\to \pm\infty$. The spectrum $\sigma(|z|)$ is contained in $[0,1)$. Clearly, for any $s\in[0,1)$,  
$$
\lim_{t\to\infty} \tanh(ts)=\left\{ \begin{array}{l} 1 \ \hbox{ if } s\in(0,1) \\ 0 \ \hbox{ if } s=0 .  \end{array}\right.
$$
By Lebesgue's bounded convergence theorem, and the Borel functional calculus for bounded selfadjoint operators, we have that
$$
\lim_{t\to\infty} \tanh(t|\alpha|)=\chi_{(0,1)}(|\alpha|)=P_{N(\alpha)^\perp}.
$$
Then, 
$$
\lim_{t\to\infty} \delta(t)=\omega P_{N(\alpha)^\perp}=\omega.
$$
Similarly, using that $\lim_{t\to-\infty} \tanh(st)=\left\{ \begin{array}{c}\  -1  \ \ \hbox{ if } s\in(0,1) \\ 0 \ \hbox{ if } s=0   \end{array}\right.$, we get that
$$
\lim_{t\to-\infty} \delta(t)=-\omega P_{N(\alpha)^\perp}=-\omega.
$$
\end{proof}

This geometric role of the partial isometry $\omega$ in the polar decomposition of $z \in \d$ (or, more generally, of every  $z \in \a\setminus\{0\}$) has not been noticed before, to the authors' knowledge.

In order to compute the limit points of arbitrary geodesics, it will be useful to extend the action of  $\u(\theta)$ to the strong operator border
$$
\partial \d:=\{a\in\a^{**}: \|a\|= 1\}
$$ 
of $\d$,  i.e., to define $\tilde{g}\cdot a$ for $a\in\a^{**}$ with $\|a\|=1$.
\begin{lem}
If $\tilde{g}\in\u(\theta)$ and  $a\in\partial \d$, then $g_{11}+g_{12}a$ is invertible in $\a^{**}$. 
\end{lem}
\begin{proof}
Note that 
$$
g_{11}+g_{12}a=\langle {\bf e_1} , \tilde{g} \left(\begin{array}{l} 1 \\ a \end{array}\right)\rangle.
$$
Using the polar decomposition $\tilde{g}=\tilde{u}|\tilde{g}|$, $\tilde{u}=\left( \begin{array}{cc} u_1 & 0  \\ 0 & u_2 \end{array} \right)$,
$$
\langle {\bf e_1} , \tilde{g} \left(\begin{array}{l} 1 \\ a \end{array}\right)\rangle=u_1^*\langle {\bf e_1} , |\tilde{g}| \left(\begin{array}{l} 1 \\ a \end{array}\right)\rangle,
$$
i.e., we may suppose $\tilde{g}\ge 0$, 
 $\tilde{g}=\left( \begin{array}{cc} (1+b^*b)^{1/2} & b^* \\ b & (1+bb^*)^{1/2} \end{array} \right)$, for $b\in\a$. Then 
$$
g_{11}+g_{12}a=(1+b^*b)^{1/2} +b^*a=(1+b^*b)^{1/2}(1+(1+b^*b)^{-1/2}b^*a).
$$
It suffices to show that $\|(1+b^*b)^{-1/2}b^*a\|<1$.
Note that, since $\|a\|=1$, 
$$
\|(1+b^*b)^{-1/2}b^*a\|^2\le \|(1+b^*b)^{-1/2}b^*\|^2=\|(1+b^*b)^{-1/2}b^*b(1+b^*b)^{-1/2}\|=\max \{f(t): t\in\sigma(b^*b)\},
$$
for $f(t)=\frac{t}{(1+t)}$. Clearly, this number is strictly less than $1$.
\end{proof}
\begin{prop}\label{accion en el borde}
If $a\in\partial \d$, and $\tilde{g}\in\u(\theta)$,  then
$$
\tilde{g}\cdot a := (g_{21}+g_{22}a)(g_{11} +g_{12}a)^{-1}\in\partial \d,
$$
defines a left action of $\u(\theta)$ on $\partial \d$.
\end{prop}

\begin{proof}
A density argument (or a proof similar as in the previous lemma), shows that if $x\in\a^{**}$ with $\|x\|<1$, then $\tilde{g}\cdot x$, defined as above, also satifies $\|\tilde{g}\cdot x\|<1$, and defines an action on the unit  ball of $\a^{**}$. Let $a\in\partial \d$. Then, by Kaplansky's density theorem, there  exists a sequence $a_n\in\d$ such that $a_n\to a$ in the strong operator topology. 
We claim that 
\begin{lem}\label{lema interno}
 $\tilde{g}\cdot a_n\to \tilde{g}\cdot a$ strongly.
\end{lem}
\begin{proof}
Consider the polar decomposition $\tilde{g}=\tilde{u}|\tilde{g}|$. We check first that $|\tilde{g}|\cdot a_n\to |\tilde{g}|\cdot a$ strongly. As before,
 $\tilde{g}=\left( \begin{array}{cc} (1+b^*b)^{1/2} & b^* \\ b & (1+bb^*)^{1/2} \end{array} \right)$. Then 
 $$
 |\tilde{g}|\cdot a_n=(b+(1+bb^*)^{1/2}a_n)((1+b^*b)^{1/2}+b^*a_n)^{-1}.
 $$
Clearly,  $b+(1+bb^*)^{1/2}a_n\to b+(1+bb^*)^{1/2}a$ and $(1+b^*b)^{1/2}+b^*a_n\to (1+b^*b)^{1/2}+b^*a$ strongly. Moreover, 
$$
((1+b^*b)^{1/2}+b^*a_n)^{-1}=(1+(1+b^*b)^{-1/2}b^*a_n)^{-1}(1+b^*b)^{-1/2}.
$$
Let us show that the norms of these inverses are uniformly bounded. It suffices to see that $\|(1+(1+b^*b)^{-1/2}b^*a_n)^{-1}\|$ are uniformly bounded.
Denote  $d_n=(1+b^*b)^{-1/2}b^*a_n$. Then, as seen above,
$$
\|d_n\|^2\le \max\{f(t): t\in\sigma(b^*b)\}=r^2<1.
$$
Thus, 
$$
\|(1+(1+b^*b)^{-1/2}b^*a_n)^{-1}\|\le \frac{1}{1-r}.
$$
Therefore the inverses $(1+(1+b^*b)^{-1/2}b^*a_n)^{-1}$ converge strongly to $(1+(1+b^*b)^{-1/2}b^*a)^{-1}$. Then, clearly, $c_n:=|\tilde{g}|\cdot a_n$ converge strongly to $c:=|\tilde{g}|\cdot a$. Since $\tilde{u}=\left(\begin{array}{cc} u_1 & 0 \\ 0 & u_2 \end{array} \right)$, it is clear that
$$
\tilde{g}\cdot a_n=\tilde{u}\cdot (|\tilde{g}|\cdot a_n)=u_2c_nu_1^*\to u_2cu_1^*=\tilde{g}\cdot a.
$$ 
\end{proof}

Let us proceed with the proof of Proposition \ref{accion en el borde}. Since $\|a_n\|<1$, we know that $\|\tilde{g}\cdot a_n\|<1$. From Lemma \ref{lema interno}, it follows that $\|\tilde{g}\cdot a\|\le 1$. Suppose that $\|\tilde{g}\cdot a\|<1$. Using again the fact that $\u(\theta)$ acts on $\d$, this would imply that 
 $$
 \tilde{g}^{-1}\cdot (\tilde{g}\cdot a)=a\in\d,
 $$
 a contradiction. Thus, $\|\tilde{g}\cdot a\|= 1$.

The fact that this rule defines, indeed,  a left action, follows from similar density arguments.
\end{proof}
Using this result, we can compute the limit points of arbitrary geodesics. Since the action of $\u(\theta)$ is transitive, given $z_1, z_2\in\d$, there exists $\tilde{g}\in\u(\theta)$ such that $\tilde{g}\cdot z_1=0$. 
\begin{coro}
Let $z_0,z_1\in\d$ and $\tilde{g}\in\u(\theta)$ such that $\tilde{g}\cdot z_0=0$. Let  $\delta$ be the unique geodesic of $\d$ such that $\delta(0)=z_0$ and $\delta(1)=z_1$. Denote by $\dot{\delta}_0$ the initial velocity of $\delta$. Then
$$
{\rm SOT}-\lim_{t\to +\infty} \delta(t)=\tilde{g}\cdot  \omega_0 \hbox{ and } {\rm SOT}-\lim_{t\to -\infty} \delta(t)=\tilde{g}\cdot  (-\omega_0),
$$
where $\omega_0\in\a^{**}$ is the partial isometry in the polar decomposition of $\dot{\delta}_0$: $\dot{\delta}_0=\omega_0|\dot{\delta}_0|$.
\end{coro}

\begin{rem}
In order to identify these limit points in $\d$,  following the notation of the above Corollary, note that if $\tilde{g}=|\tilde{g}^*|\tilde{v}= \left(\begin{array}{cc} (1+b^*b)^{1/2} & b^* \\ b & (1+bb^*)^{1/2} \end{array}\right)\left(\begin{array}{cc} v_1 & 0 \\ 0 & v_2 \end{array}\right)$ (the reversed polar decomposition),  then $\tilde{v}\cdot\omega_0=v_2\omega_0v_1^*$ is a  partial isometry. Therefore,  the limit points of geodesics  are elements in $\partial \d$ of the form
$$
(b+(1+bb^*)^{1/2}\omega)((1+b^*b)^{1/2}+b^*\omega)^{-1} \hbox{ and }  (b-(1+bb^*)^{1/2}\omega)((1+b^*b)^{1/2}-b^*\omega)^{-1},
$$ 
where $b\in\a$ is arbitrary and $\omega\in\a^{**}$ is a partial isometry.

Note that not any partial isometry in $\a^{**}$ occurs in the polar decomposition of an element in $\d$. For instance, if $\a=C([0,1])$ (continuous functions in the unit interval), the polar decomposition of $f\in\a$ is $f=w|f|$, where
$w\in L^\infty(0,1)$ is given by $w(t)=\left\{ \begin{array}{l} f(t)/|f(t)| \hbox{ if } f(t)\ne 0 \\ 0 \hbox{ if } f(t)=0  \end{array} \right. $. An arbitrary partial isometry in $L^\infty(0,1)$ is a measurable function whose values are zero or complex numbers of modulus $1$. The set of zeros of such a function is an arbitrary measurable set, whereas the set of zeros of partial isometries which occur in the polar decomposition of a continuous function, are closed subsets of $[0,1]$.

Another way to study the limit points of geodesics, is by using the Borel subgroup $\b_\theta\subset\u(\theta)$ instead. Indeed, since the action of this group is transitive in $\d$, any limit point of a geodesic is either of the form $\tilde{g}\cdot v$ or $\tilde{g}\cdot(-v)$, for $\tilde{g}\in\b_\theta$. Consider the following example:
\begin{ejem}
Suppose that $\a$ is a von Neumann algebra, and let $p\ne 0$ be a projection in $\a$. For $\tilde{g}=\left( \begin{array}{cc} \displaystyle{\frac{g+\hat{g}}{2}}-\hat{g}x & \displaystyle{\frac{g-\hat{g}}{2}}-\hat{g}x \\ \displaystyle{\frac{g-\hat{g}}{2}}+\hat{g}x & \displaystyle{\frac{g+\hat{g}}{2}}+\hat{g}x  \end{array}\right)$, let us compute $\tilde{g}\cdot p$. After straightforward computations,
$$
\tilde{g}\cdot p=\left( g(1+p)+\hat{g}\left(-1+p+2x(1+p)\right)\right)\left(g(1+p)+\hat{g}\left(1-p -2x(1+p)\right)\right)^{-1}.
$$
Note that $1+p$ is invertible and that $(1-p)(1+p)^{-1}=1-p$. Then
$$
\tilde{g}\cdot p=\left(1+\hat{g}\left(p-1+2x(1+p)\right) (1+p)^{-1} g^{-1}\right) \left(1+\hat{g}\left(1-p-2x(1+p)\right) (1+p)^{-1}g^{-1}\right)^{-1}=
$$
$$
=\left(1+\hat{g}(p-1+2x)g^{-1}\right)\left(1+\hat{g}(1-p-2x)g^{-1}\right)^{-1}.
$$
Denote $\alpha=\hat{g}(p-1)g^{-1}$ and $\beta=2\hat{g}xg^{-1}$. 
Observe that $\alpha$ is a non-invertible selfadjoint element, $\alpha\le 0$ and its range is proper and closed; $\beta$ is an arcitrary anti-selfadjoint element. Then
$$
\tilde{g}\cdot p=(1+\alpha+\beta)\left(1-(\alpha+\beta)\right)^{-1}.
$$
Note that $1-(\alpha+\beta)$ is invertible because  $Re(1-(\alpha+\beta))=1-\alpha\ge 1$. If one picks $p=1$, then $\alpha=0$ and
$$
\tilde{g}\cdot 1=(1+\beta)(1-\beta)^{-1},
$$
which is a unitary operator such that $-1$ does not belong to its spectrum. In particular, this shows that the action of $\b_\theta$ ceases to be transitive in $\partial\d$.

\end{ejem}

\end{rem}

Our next result shows a necessary condition for an element $a\in\a^{**}$ with $\|a\|=1$ to be a limit point of a geodesic of $\d$

\begin{prop}
If $a\in\a^{**}$ is the limit point at $+\infty$  of a geodesic in $\d$, then $1-a^*a=hqh$, where $h\in G^+$ and $q\in\a^{**}$ is a projection. In particular, not every element of norm $1$ in $\a^{**}$ is the limit point of a geodesic: such elements satisfy that the defect element $1-a^*a$ has closed range.
\end{prop}
\begin{proof}
If $a$ is the limit point of a geodesic if and only if there exists $\tilde{g}\in\u(\theta)$ and a partial isometry $\omega\in\a^{**}$ such that $a=\tilde{g}\cdot a$. The definition of the action implies that this equality can be read as an usual matrix equality
$$
\tilde{g} \left( \begin{array}{c} 1 \\ \omega \end{array}\right)=\left( \begin{array}{c} b_1 \\ b_2 \end{array}\right),
$$
with $b_1\in G$ and $b_2b_1^{-1}=a$.
Using the form $\theta$, and the fact that $\tilde{g}\in\u(\theta)$,
$$
\theta(\left( \begin{array}{c} b_1 \\ b_2 \end{array}\right),\left( \begin{array}{c} b_1 \\ b_2 \end{array}\right))=\theta(\left( \begin{array}{c} 1 \\ \omega \end{array}\right),\left( \begin{array}{c} 1 \\ \omega \end{array}\right)),
$$
i.e. 
$$
b_1^*b_1-b_2^*b_2=1-\omega^*\omega=q',
$$
which is a projection in $\a^{**}$. Let $b_1=u|b_1|$ be the polar decomposition, with $u$ unitary. Then 
$au=b_2b_1^{-1}u=b_2|b_1|^{-1}$. Thus,
$$
q'=b_1^*b_1-b_2^*b_2=|b_1|(1-|b_1|^{-1}b_2^*b_2|b_1|^{-1})|b_1|=,|b_1|(1-u^*a^*au)|b_1|,
$$
i.e.,
$$
1-a^*a=u|b_1|^{-1}q'|b_1|^{-1}u^*=hqh,
$$
for $q=uq'u^*$ a projection in $\a^{**}$ and $h=u|b_1|^{-1}u^*\in G^+$.
\end{proof}

\begin{rem}
We believe that the characterization of the partial isometries which appear in the polar decompositions of all limit points $a \in \a^{**}$ is an interesting open problem.
\end{rem}

\section{The invariant metric in $\pa^\theta$}\label{seccion inner product}

The characterization of the tangent spaces done in (\ref{tangente}), in Section \ref{seccion corta}, enables us to define
a Finsler metric, that is, a continuous distribution of norms on the tangent bundle of $\pa^\theta$.

Given $\ell\in\pa^\theta$ and $V\in(T\pa^\theta)_\ell$, fix a generator $\xx_0\in\k_\theta$ for $\ell$, i.e., $[\xx_0]=\ell$ and $\theta(\xx_0,\xx_0)=1$. Recall, from Remark \ref{como cambia}, that $\xx_0$ is determined up to a unitary element of $\a$: if $\xx'_0$ is another such generator, then there exists $u\in\u_\a$ such that $\xx'_0=\xx_0 u$. Having fixed a generator for $\ell$, as we saw in Section \ref{seccion corta}, $(T\pa^\theta)_\ell$ identifies with $\ell^{\perp_\theta}$, and to the tangent vector $V$ corresponds an element $\vv\in\ell^{\perp_\theta}$. We define:
\begin{equation}\label{finsler}
|V|_\ell:=\|\theta(\vv,\vv)\|^{1/2}.
\end{equation}
Note that $|\ \ |_\ell$ does not depend on the choice of the generator. If we choose $\xx'_0=\xx_0 u$ instead, the tangent vector $V$ is represented by $\vv'=\vv u\in\ell^{\perp_\theta}$, and therefore
$$
\|\theta(\vv',\vv')\|=\|\theta(\vv u,\vv u)\|=\|u^*\theta(\vv,\vv)u\|=\|\theta(\vv,\vv)\|.
$$
Next, recall from Lemma \ref{forma negativa}, that $\theta$ is negative definite (non degenerate), and therefore the expression (\ref{finsler}) above defines a proper norm in $(T\pa^\theta)_\ell$.

\begin{rem}
If $V\in (T\pa^\theta)_{[\xx_0]}$, by the identification in (\ref{tangente}), Section \ref{seccion corta},  $V$ is represented by some $\vv\in[\xx_0]^{\perp_\theta}$; since $\yy_0=\left( \begin{array}{c} (x_1^*)^{-1}x_2 \\ 1 \end{array} \right)$, for $\xx_0=\left( \begin{array}{c} x_1 \\ x_2 \end{array} \right)$, is a generator of $[\xx_0]^{\perp_\theta}$, there exists $a\in\a$ such that $\vv=\yy_0 a$. 
Then
$$
|V|_{[\xx_0]}=\|\theta(\yy_0a,\yy_0a)\|^{1/2}=\|a^*(1-x_2x_1^{-1})^*a \|^{1/2}=\|(1-|(x_2x_1^{-1})^*|^2)^{1/2}a\|.
$$
Note also that the norm of $\vv=\yy_0 a$ in $\a^2$ is
$$
\|\yy_0 a\|=\|\langle \yy_0 a,\yy_0 a\rangle\|^{1/2}=\|(1+|(x_2x_1^{-1})^*|^2)^{1/2}a\|.
$$
\end{rem}

\begin{prop}
For any $\ell\in\pa^\theta$, the norm $| \ |_{\ell}$ of $(T\pa^\theta)_{\ell}$ is complete. 
\end{prop}
\begin{proof}
With the current notations, if $V=\yy_0 a\in(T\pa^\theta)_{[\xx_0]}$,
$$
|V|_{[\xx_0]}=\|(1-|(x_2x_1^{-1})^*|^2)^{1/2}a\|=\|(1-|(x_2x_1^{-1})^*|^2)^{1/2}(1+|(x_2x_1^{-1})^*|^2)^{-1/2}(1+|(x_2x_1^{-1})^*|^2)^{1/2}a\|
$$
$$
\le \|\{(1-|(x_2x_1^{-1})^*|^2)(1+|(x_2x_1^{-1})^*|^2)^{-1}\}^{1/2}\| \|\yy_0a\|.
$$
Similarly
$$
\|\yy_0a\|\le \|\{(1+|(x_2x_1^{-1})^*|^2)(1-|(x_2x_1^{-1})^*|^2)^{-1}\}^{1/2}\| |V|_{[\xx_0]}.
$$
It follows that on $(T\pa^\theta)_{\ell]}\simeq [\yy_0]$, the metric $|\ |_{\ell}$ and the norm of $\a^2$ are equivalent. Since $[\yy_0]$ is closed in $\a^2 $, it is complete, and the proof follows.
\end{proof}

The distribution $\pa^\theta \ni\ell\mapsto |\ \ |_\ell$  is clearly continuous. Thus,  $\pa^\theta$ is endowed with a Finsler metric.

The following result is tautological, but of the utmost importance for our discussion:

\begin{teo}
The Finsler metric defined in (\ref{finsler}) is invariant under the action of $\u(\theta)$.
\end{teo}
\begin{proof}
Pick $\ell=[\xx_0]\in\pa^\theta$,  $V\in(T\pa^\theta)_{\ell}$ (as before, $V$ is identified to $\sim\yy_0 a$) and $\tilde{g}\in\u(\theta)$. The action of $\u(\theta)$ on $\pa^\theta$ induces an action on the tangent spaces. As quotients: $\tilde{g}$ maps $\a^2 /[\xx_0]$ onto $\a^2 / [\tilde{g}\xx_0]$, because $\tilde{g}$, being $\a$-linear, maps $[\xx_0]$ onto $[\tilde{g}\xx_0]$. But as $\theta$-orthogonal submodules as well: since $\tilde{g}$ preserves $\theta$, 
$$
\tilde{g}([\yy_0]^{\perp_\theta})=[\tilde{g}\yy_0]^{\perp_\theta}.
$$
Then
$$
|\tilde{g}V|_{\tilde{q}[\xx_0]}=\|\theta(\tilde{g}(\yy_0 a),\tilde{g}(\yy_0a))\|^{1/2}=\|\theta(\yy_0 a,\yy_0 a)\|^{1/2}=|V|_{[\xx_0]}.
$$ 
\end{proof}

\subsection{Invariant metric in $\d$}

We need to compute the differential of the map $\pa^\theta\to\d$ at $[{\bf e_1}]$. To do so, we use the above diagram (\ref{diagrama}).
Recall that $(T\pa^\theta)_{[\xx_0]}=\a^2 / \{\xx_0 a: a\in\a\}$; in particular 
$$
(T\pa^\theta)_{[{\bf e_1}]}=\a^2 / \a\times \{0\}.
$$
Thus, any tangent element $V\in(T\pa^\theta)_{[{\bf e_1}]}$ has a unique representative $V=\left[\left( \begin{array}{c} 0 \\ x \end{array} \right)\right]$, for $x\in\a$.
\begin{lem}\label{diferencial}
The differential of the map $\pa^\theta\to\d$, $[\xx]\mapsto x_2x_1^{-1}$,  at $[{\bf e_1}]$ is the map
$$
\left[\left( \begin{array}{c} 0 \\ x \end{array} \right)\right]\longmapsto x \ , \ \ x\in\a.
$$
\end{lem}
\begin{proof}
Fix $x\in\a$. We use the commutative diagram (\ref{diagrama}). Let $\xx(t)\in\k_\theta$ be a smooth curve such that $\xx(0)={\bf e_1}$, and the derivative of $[\xx(t)]$ (in $\pa^\theta$) at $t=0$ is $\left[\left( \begin{array}{c} 0 \\ x \end{array} \right)\right]$. Then $x_1(0)=1, x_2(0)=0, \dot{x}_2(0)=x$. If we map $\xx(t)$ onto $\d$, and differentiate at $t=0$ we get:
 $$
\frac{d}{dt} x_2(t)x_1^{-1}(t)|_{t=0}= \dot{x}_2(0)x_1^{-1}(0)-x_2(0)x_1^{-1}(0)\dot{x}_1(0) x_1^{-1}(0)=x,
$$
which proves our claim.
\end{proof}

\begin{teo}
The identification $[\xx]\longleftrightarrow x_2x_1^{-1}$ between $\pa^\theta$ and $\d$, is isometric.
\end{teo}
\begin{proof}
The group $\u(\theta)$ acts isometrically on both $\pa^\theta$ and $\d$. The action is isometric in both spaces. Therefore, it suffices to show that the differential at $[{\bf e_1}]$ is an isometry.  By Lemma \ref{diferencial}, this map is
$$
\left[\left( \begin{array}{c} 0 \\ x \end{array} \right)\right]\longmapsto x.
$$
The norm of $\left[\left( \begin{array}{c} 0 \\ x \end{array} \right)\right]$ is computed via the identification $(T\pa^\theta)_{[{\bf e_1}]}\simeq [{\bf e_1}]^{\perp_\theta}=[{\bf e_2}]$, which sends $\left[\left( \begin{array}{c} 0 \\ x \end{array} \right)\right]$ to ${\bf e_2 x}$. The norm (in $\a^2$) of this element is 
$$\|x^*\theta({\bf e_2},{\bf e_2})x\|^{1/2}=\|x^*x\|^{1/2}=\|x\|.
$$
On the other hand, the norm of $x$ as a tangent element of $\d$ at $0$ is the usual norm $\|x\|$ (in $\a$).
\end{proof}

 \subsection{The  action of the Borel subgroup $\b_\theta$}

As we mentioned in Proposition \ref{tejemas} and Remark \ref{tejemas}, the group $\b_\theta$, which we call the Borel group of the form $\theta$, is a subgroup of $\u(\theta)$ which acts transitively in $\d$. Since it is a subgroup of $\u(\theta) $, the action is also isometric (recall also that the action is free in the hyperboloid $\k_\theta$). Therefore it acts transitively and isometrically in $\pa^\theta$. 

Given $z\in\d$, let us find an element $\tilde{g}\in\b_\theta$ such that $\tilde{g}\cdot 0=z$. First, we describe the action of $\u(\theta)$ on $\d$, which factors through the hyperboloid $\k_\theta$:
\begin{rem}
Given $\tilde{g}\in\u(\theta)$ and $z\in\d$, we lift $z$ to $\k_\theta$ by means of the global section
$$
\d\ni z\mapsto \left( \begin{array}{c} (1-z^*z)^{-1/2} \\  z(1-z^*z)^{-1/2} \end{array} \right).
$$
Next, we multiply $\tilde{g} \left( \begin{array}{c} (1-z^*z)^{-1/2} \\  z(1-z^*z)^{-1/2} \end{array} \right)$, and finally we compose with the fibration 
$$
\k_\theta\ni \vecz \mapsto z_2z_1^{-1}\in\d.
$$
\end{rem}
\begin{defi}
If $z\in \d$, we define
$$
\tilde{g}_z=\left( \begin{array}{cc} (1-z^*z)^{-1/2} & (1+z^*)^{-1}z^*(z+1)(1-z^*z)^{-1/2} \\ z(1-z^*z)^{-1/2} & (1+z^*)^{-1}(z+1)(1-z^*z)^{-1/2} \end{array}\right)
$$
$$
=\left( \begin{array}{cc} (1+z^*)^{-1} & 0 \\ 0 & (1+z^*)^{-1} \end{array}\right) 
\left( \begin{array}{cc} 1+z^* & z^*(z+1) \\ (1+z^*)z & z+1 \end{array}\right)
\left( \begin{array}{cc} (1-z^*z)^{-1/2} & 0 \\ 0 & (1-z^*z)^{-1/2} \end{array}\right).
$$
\end{defi}

Observe that the diagonal matrices on the right and left hand sides  do not belong to $\u(\theta)$, so that this is not a proper factorization of $\tilde{g}_z$.

\begin{lem}\label{matrizon}
If $z\in \d$, then $\tilde{g}_z\in\b_\theta$ and satisfies $\tilde{g}_z\cdot 0=z$.
\end{lem}
\begin{proof}
The origin $0\in\d$ lifts to ${\bf e_1}\in\k_\theta$. Recall from Remark \ref{tejemasbis} the form of the elements in $\b_\theta$: 
$$
\b_\theta= \{ \left(  \displaystyle{\begin{array}{cc} \displaystyle{\frac{g+\hat{g}}{2}} -\hat{g} x & \displaystyle{\frac{g-\hat{g}}{2}}-\hat{g}x \\ \displaystyle{\frac{g-\hat{g}}{2}}+\hat{g}x & \displaystyle{\frac{g+\hat{g}}{2}} +\hat{g}x \end{array}} \right): g\in G ,  x^*=-x \} .
$$
Thus, we are looking for $g\in G$ and $x\in\a$ with $x^*=-x$ such that
$$
\frac12(g+\hat{g})-\hat{g}x=(1-z^*z)^{-1/2} \ \hbox{ and } \ \frac12(g-\hat{g})+\hat{g}x=z(1-z^*z)^{-1/2}.
$$
That is, $g=(1+z)(1-z^*z)^{-1/2}$, thus $\hat{g}=(1+z^*)^{-1}(1-z^*z)^{1/2}$ and
$$
\hat{g}x=\frac12(1+z^*)^{-1}\{z-z^*\}(1-z^*z)^{-1/2}.
$$
 Then, 
$$
x=\frac12 g^*(1+z^*)^{-1}\{z-z^*\}(1-z^*z)^{-1/2}=\frac12 (1-z^*z)^{-1/2}(z-z^*)(1-z^*z)^{-1/2},
$$
which is clearly anti-Hermitian, and thus $\tilde{g}_z\in\b_\theta$.
\end{proof}
\begin{rem}\label{1/2}
Denote by the $d$ the distance defined by the the Finsler  metric introduced in $\d$. In \cite{tejemas} (Section 8) it was shown that
$$
d(0,z)=\frac12\log( \frac{1+\|z\|}{1-\|z\|}).
$$
Using the fact that the action of $\b_\theta$  on $\d$ is isometric, and the above construction of $\tilde{g}_z$, for arbitrary $z_1,z_2\in\d$, we can  compute  $d(z_1,z_2)$ as follows:
\begin{equation}\label{distancia disco}
d(z_1,z_2)=d(0,\tilde{g}_{z_1}^{-1}\cdot z_2)=\frac12\log( \frac{1+\|\tilde{g}_{z_1}^{-1}\cdot z_2\|}{1-\|\tilde{g}_{z_1}^{-1}\cdot z_2\|}).
\end{equation}
In order to compute $\tilde{g}_{z}^{-1}$, recall that elements $\tilde{g}$ in $\u(\theta)$ are characterized by the relation $\rho \tilde{g}^*\rho=\rho^{-1}$. Then
$$
\tilde{g}_z^{-1}=\left( \begin{array}{cc} (1-z^*z)^{-1/2} & 0 \\ 0 & (1-z^*z)^{-1/2} \end{array}\right)
\left( \begin{array}{cc} 1+z & -z^*(1+z) \\ -(1+z^*)z & 1+z^* \end{array}\right)
\left(\begin{array}{cc} (1+z)^{-1} & 0 \\ 0 & (1+z)^{-1} \end{array}\right).
$$
Then, after straightforward computations,
$$
\tilde{g}_{z_1}^{-1}\cdot z_2=(1-z_1^*z_1)^{-1/2}(1+z_1^*)(1+z_1)^{-1}(z_2-z_1)(1-z_1^*z_2)^{-1}(1-z_1^*z_1)^{1/2}.
$$
\end{rem}

\section{Operator cross ratio in the  hyperbolic part of the projective line}

Here we state our main result, relating the metric of $\pa^\theta$  introduced in Section 8, with the so called operator cross ratio, as defined in the Grassmann manifold of a Hilbert space by M.I. Zelikin \cite{zelikin}. We shall apply these ideas to the rank one submodules in $\pa^\theta$. To this effect, the isometry between $\pa^\theta$ and the disk $\d$ will be important. 

Consider $\ell=\left[ \left(\begin{array}{l} 1 \\ z \end{array} \right) \right]\in\pa^\theta$, for $z\in\d$. Let $z=\omega|z|$ be the polar decomposition. Let $\delta$ be the geodesic of $\pa^\theta$ such that $\delta(0)=\left[ \left(\begin{array}{l} 1 \\  0 \end{array} \right) \right]$ and $\delta(1)=\left[ \left(\begin{array}{l} 1 \\ z \end{array} \right) \right]$. Equivalently, regarded in $\d$: $\delta(0)=0$ and $\delta(1)=z$.  As seen in Section 6, 
$$
{\rm SOT}-\lim_{t\to+\infty} \delta(t)=\omega \hbox{ and } {\rm SOT}-\lim_{t\to-\infty} \delta(t)=-\omega.
$$
Four points are determined: $-\omega, 0 , z, \omega$, or better, four submodules 
$$
\ell_{-\infty}:=\left[ \left(\begin{array}{c} 1 \\ -\omega \end{array} \right) \right], \ell_0:=\left[ \left(\begin{array}{l} 1 \\ 0 \end{array} \right) \right], \ell=\left[ \left(\begin{array}{l} 1 \\ z \end{array} \right) \right], \ell_{+\infty}:=\left[ \left(\begin{array}{l} 1 \\ \omega \end{array} \right) \right],
$$
where the limit lines lie in $\partial \pa^\theta$.

In Section 3  we defined the operator cross ratio of four elements in $\pa$, as a (possibly empty) set of module endomorphisms, following ideas in \cite{zelikin}. Here we compute the  operator cross ratio $CR(\ell_{-\infty},\ell_0,\ell, \ell_{+\infty})$,  proving that it is non empty, and  that there exists a natural $\ell$-endomorphism to choose from this set. 

Recall that elements of $CR(\ell_{-\infty},\ell_0,\ell, \ell_{+\infty})$ are  (module) endomorphisms of $\ell$, defined  as the composition of 
{\it the projection from $\ell$ to $\ell_0$ parallel  to $\ell_{-\infty}$, followed by the projection from $\ell_0$ to $\ell$ parallel to $\ell_{+\infty}$}.

In coordinates, by choosing generators in the respective submodules
$$
\left(\begin{array}{c} 1 \\ z \end{array} \right) \mapsto \left(\begin{array}{c} 1 \\ 0 \end{array} \right)\lambda , \ \left(\begin{array}{c} 1-\lambda \\ z \end{array} \right)= \left(\begin{array}{c} 1 \\ -\omega \end{array} \right) \mu.
$$
Then $1-\lambda=\mu$ and $z=-\omega\mu$. Then $\omega|z|=-\omega\mu$. If $z$ is invertible (and then $\omega$ is unitary) this implies $\mu=-|z|$, otherwise this is just one possible solution. Non uniqueness of solutions of these equations reflect the geometric fact that the modules  $\ell_0$ and $\ell_\infty$ may not be in direct sum. Explicitly, all solutions of these equations are of the form
$$
\lambda=1+|z|-\Omega \ , \ \ \mu=-|z|+\Omega,
$$
where $\Omega\in\a^{**}$ is  such that $\omega\Omega=0$. In particular, $|z|\Omega=|z|\omega^*\omega\Omega=0$. We choose the solution with $\Omega=0$. Note that $\lambda=1+|z|$, and therefore the first projection in the above composition is given by
$$
\left(\begin{array}{c} 1 \\ z \end{array} \right)\mapsto \left(\begin{array}{c} 1 +|z| \\ 0 \end{array} \right).
$$
Next 
$$
\left(\begin{array}{c} 1+|z| \\ 0 \end{array} \right)\mapsto \left(\begin{array}{c} 1 \\ z \end{array} \right)\gamma \ , \ \left(\begin{array}{c} 1+|z|-\gamma \\ -z\gamma \end{array} \right)=\left(\begin{array}{c} 1 \\ \omega \end{array} \right)\epsilon.
$$
So that $1+|z|-\gamma=\epsilon$ and $-z\gamma=\omega\epsilon$, and then  (the unique solution if $\omega$ is unitary, or a possible solution that we choose, otherwise)  $1+|z|-\gamma=-|z|\omega$, i.e.,
$$
\gamma=(1-|z|)^{-1}.
$$
Other solutions of the above equation are of the form
$$
\gamma=(1-|z|)^{-1}+(1-|z|)^{-1}\Omega'    , 
$$
where $\Omega'\in\a^{**}$ is  such that $\omega\Omega'=0$.
In general, the possible endomorphisms $\ell\to\ell$ are given  (in these coordinates) by
$$
\left(\begin{array}{c} 1 \\ z \end{array} \right)\mapsto \left(\begin{array}{c} 1 \\ z \end{array} \right) (1+|z|+\Omega)(1-|z|)^{-1}(1+\Omega')= (1+|z|)(1-|z|)^{-1}+\Omega' +\Omega(1-|z|)^{-1}+\Omega\Omega',
$$
where we use that  $\omega\Omega=\omega\Omega'=0$, and thus $(1\pm |z|)^{\pm 1}\Omega=\Omega$ (and the same for $\Omega'$).

As noted, if $z$ is invertible, there is a unique solution with $\Omega=\Omega'=0$. Our choice of cross ratio, picking $\Omega=\Omega'=0$ in any case,  is justified below. 
First, we prove the following fact, which  is certainly well known.
\begin{lem}\label{modulo sot}
Let $a_n\in\a$ with $\|a_n\|\le 1$. If $a_n\to a$ strongly, then $|a_n|\to |a|$ strongly.
\end{lem}
\begin{proof}
For any $\xi\in\h$,
$$
\||a_n|^2\xi-|a|^2\xi\|^2=\|a_n\xi\|^4+\|a\xi\|^4-2 Re \langle a_n\xi,a_n|z|^2\xi\rangle.
$$
Clearly $\|a_n\xi\|\to \|a\xi\|$ and $\langle a_n\xi,a_n|z|^2\xi\rangle\to \langle a\xi,a|a|^2\xi\rangle=\|a\xi\|^4$, so that $|a_n|^2\to |a|^2$ strongly. It is known that the square root of positive operators  is strongly continuous in the unit ball: if $0\le b_n\le 1$ and $b_n\to b$ strongly, then $b_n^{1/2}\to b^{1/2}$. Let us also sketch a proof of this fact. If $\xi\in\h$,
$$
\|b_n^{1/2}\xi-b^{1/2}\xi\|^2=\langle b_n\xi,\xi\rangle+\langle b\xi,\xi\rangle-2 Re\langle b_n^{1/2}\xi,b^{1/2}\xi \rangle ,
$$
and 
$$
\langle ba_n^{1/2}\xi,b^{1/2}\xi \rangle=\int_0^\infty t^{1/2} d\mu_n(t),
$$
where $\mu_n=\mu_{b_n, \xi,b^{1/2}\xi}$ is the scalar spectral measure of $b_n$, associated to the pair of vectors $\xi, b^{1/2}\xi$.  If $p(t)$ is a polynomial, than clearly $p(b_n)\to p(b)$ strongly, because $\|b_n\|\le 1$. Then $\mu_n$ converge to $\mu$ (the scalar spectral measure of $a$ associated the the vectors $\xi, b^{1/2}\xi$), when, regarded as functionals in $C(0,1)^*$, they are evaluated at polynomials. Since the norms of these measures are uniformly bounded,
$$
\|\mu_n\|\le \|\xi\|\|b_n^{1/2}\xi\|\le \|\xi\|^2,
$$
it follows that $\mu_n(t^{1/2})$ converges weakly to $\mu(t^{1/2})$.
\end{proof}

\begin{rem}
 Dixmier and Marechal \cite{dixmar} proved that the set on invertible elements of a von Neumann algebra is strong operator dense in the algebra. 
The argument in \cite{dixmar}  proceeds as follows. Let $a=u|a|$ be the polar decomposition of $a$ ($u\in\a^{**}$). First, the algebra $\a^{**}$ is factored in its finite and properly infinite parts.  In the finite part $u$ can be chosen unitary.  In the properly infinite part, one readily sees that it suffices to consider the cases in which $u$ is an isometry or a co-isometry. In the case that $u$ is an isometry, Dixmier and Mar\'echal prove  that $u$ is the strong limit of unitaries $u_n$.  If $u$ is a co-isometry, they show that there exist invertible elements $g_n$ which converge strongly to $u$, with norms $\|g_n\|=1$ (this is clear in the proof, though it is not stated in their result).
Summarizing, if $\a$ is a von Neumann algebra, and $a\in \a$, there exist $g_n\in G$ with $\|g_n\|\le \|a\|$ such that 
$$
{\rm SOT}-\lim_{n\to\infty} g_n =a.
$$
\end{rem}

Using this result, it is clear that, if $\a$ is a von Neumann algebra, and $z\in\d$, then there exists  and $z_n\in G$ with $\|z_n\|\le \|z\|$,  such that $z_n\to z$ strongly.

\begin{prop}\label{cont sot}
Let $z_n,z\in\d$ such that $z_n\to z$ strongly and $\|z_n\|\le \|z\|$. Then
$$
(1+|z_n|)(1-|z_n|)^{-1}\to (1+|z|)(1-|z|)^{-1} \hbox{ strongly}.
$$
\end{prop}
\begin{proof}
First note that $|z_n|\to |z|$ strongly. 
Then $1\pm |z_n|$ converges strongly to $1\pm |z|$. Clearly $\|1+|z_n|\|\le 2$. Let us check that also $\|(1-|z_n|)^{-1}\|$ are uniformly bounded:
$$
\|(1-|z_n|)^{-1}\|=\|\sum_{k=0}^\infty |z_n|^k\|\le \sum_{k=0}^\infty \|z_n^k\|\le \frac{1}{1-\|z\|}<\infty.
$$
Therefore, $(1-|z_n|)^{-1}\to (1-|z|)^{-1}$ strongly, and since the product is strongly continuous on norm bounded sets, the proof follows.
\end{proof}
\begin{defi}
Let  $z\in\d$. We define $cr(0,z)\in  CR(\ell_{-\infty},\ell_0,\ell_z, \ell_\infty)$,  for $\ell_z=\left[\left( \begin{array}{c} 1 \\ z \end{array}\right)\right]$, as the endomorphism
$$
cr(0,z):\ell_z\to \ell_z\ , \ \   cr(0,z)\big( \left( \begin{array}{c} 1 \\ z \end{array}\right)a\big)= \left( \begin{array}{c} 1 \\ z \end{array}\right) (1+|z|)(1-|z|)^{-1} a,
$$
for $a\in\a$.
\end{defi}
We use the action of $\u(\theta)$ to extend this definition to any pair $z_0\ne z_1\in\d$.
\begin{defi}\label{cr general}
Let $z_0,z_1\in\d$, $z_0\ne z_1$. Pick $\tilde{g}\in\u(\theta)$ such that $\tilde{g}\cdot 0=z_0$ and denote $z=\tilde{g}^{-1}\cdot z_1$.
We define
$$
cr(z_0,z_1)=\tilde{g}\  cr(0,z) \tilde{g}^{-1}.
$$
\end{defi}
Before checking that the definition does not depend on the choice of $\tilde{g}$, we remark the following.
Let $\delta$ be the unique geodesic of $\d$ such that $\delta(0)=z_0$ and $\delta(1)=z_1$, and let 
$$
z_{-\infty}={\rm SOT}-\lim_ {t\to -\infty}\delta(t) \ \hbox{ and }  z_{+\infty}={\rm SOT}-\lim_ {t\to+\infty}\delta(t).
$$
Then $cr(z_0,z_1)\in CR(\ell_{z_{-\infty}}, \ell_{z_0}, \ell_{z_1}, \ell_{z_{+\infty}})$, because $\tilde{g}$ is a module homomorphism which maps $\ell_z$ onto $\ell_{z_1}$. Indeed,
if $\xx=\left(\begin{array}{c} 1 \\ z \end{array}\right) a\in\ell_z$, then clearly
$$
\tilde{g}\xx=\left(\begin{array}{c} 1 \\ \tilde{g}\cdot z \end{array}\right) (g_{11}+g_{12}z)a=\left(\begin{array}{c} 1 \\ z_1 \end{array}\right)(g_{11}+g_{12}z)a\in \ell_{z_1}.
$$
 Let us check that $cr(z_1,z_2)$ is well defined, i.e.,  that it does not depend on the choice of $\tilde{g}$. 

To prove this, recall from Remark \ref{isotropiaD} that  if $\tilde{k}\in\u(\theta)$ satisfies $\tilde{k}\cdot 0=0$, then
$\tilde{k}=\left( \begin{array}{cc} u_1 & 0 \\ 0 & u_2 \end{array} \right)$, with $u_1,u_2\in\u_\a$.
\begin{prop}
With the above notations, the endomorphism 
$$
cr(z_0,z_1)\in  CR(\ell_{z_{-\infty}}, \ell_{z_0}, \ell_{z_1}, \ell_{z_{+\infty}})
$$
does not depend on the choice of $\tilde{g}$. Namely, if $\tilde{h}\in\u(\theta)$ satisfies $\tilde{h}\cdot 0 =z_0$, and $z'=\tilde{h}^{-1}\cdot z_1$, then
$$
\tilde{h} \ cr(0,z') \tilde{h}^{-1}= \tilde{g} \ cr(0,z) \tilde{g}^{-1}.
$$
\end{prop}
\begin{proof}
Since $\tilde{h}\cdot 0=\tilde{g}\cdot 0$, it follows that $(\tilde{g}^{-1}\tilde{h})\cdot 0 = 0$, and therefore
$\tilde{h}=\tilde{g} \left(\begin{array}{cc} u_1 & 0 \\ 0 & u_2 \end{array} \right)$, for $u_1,u_2\in\u_\a$. Then 
$$
z'=\tilde{h}^{-1}\cdot z_1= \left(\begin{array}{cc} u_1^* & 0 \\ 0 & u_2^* \end{array} \right)\cdot (\tilde{g}^{-1}\cdot z_1)=\left(\begin{array}{cc} u_1^* & 0 \\ 0 & u_2^* \end{array} \right)\cdot z= u_2^*zu_1,
$$
and
$$
\tilde{h} \ cr (0,z') \tilde{h}^{-1}=\tilde{g} \ \left(\begin{array}{cc} u_1 & 0 \\ 0 & u_2 \end{array} \right) \ cr (0, u_2^*zu_1) \left(\begin{array}{cc} u_1^* & 0 \\ 0 & u_2^* \end{array} \right) \tilde{g}^{-1}.
$$
Thus, we must show that $\left(\begin{array}{cc} u_1 & 0 \\ 0 & u_2 \end{array} \right) \ cr (0, u_2^*zu_1) \left(\begin{array}{cc} u_1^* & 0 \\ 0 & u_2^* \end{array} \right)=cr(0,z)$.
Let us see  how the left hand side endomorphism transforms the element $\left( \begin{array}{c} 1 \\ z \end{array} \right) a\in\ell_z$. First, it is sent to  
$$
\left(\begin{array}{cc} u_1^* & 0 \\ 0 & u_2^* \end{array} \right)\left( \begin{array}{c} 1 \\ z \end{array} \right)a=\left( \begin{array}{c} 1 \\ u_2^*zu_1 \end{array} \right)u_1^*a.
$$
The map $cr(0,z')$ maps this element to
$$
\left(\begin{array}{c} 1 \\ u_2^*zu_1 \end{array} \right)(1+|u_2^*zu_1|)(1-|u_2^*zu_1|)^{-1} u_1^*a.
$$
Note that $|u_2^*zu_1|=((u_2^*zu_1)^* u_2^*zu_1)^{1/2}=(u_1^*z^*zu_1)^{1/2}=u_1^*|z|u_1$. Therefore, the above element equals
$$
\left(\begin{array}{c} u_1^* \\ u_2^*z \end{array} \right)(1+|z|)(1-|z|)^{-1} a.
$$
Finally, multiplying on the left by the matrix $\left(\begin{array}{cc} u_1 & 0 \\ 0 & u_2 \end{array} \right)$ yields
$$
\left(\begin{array}{c} 1 \\ z \end{array} \right) (1+|z|)(1-|z|)^{-1}a=cr(0,z)\left(\begin{array}{c} 1 \\ z \end{array} \right)a.
$$
\end{proof}
\begin{rem}
If $\a$ is a von Neumann algebra, by the result of Dixmier and Mar\'echal \cite{dixmar}, $\d\cap G_\a$ is strongly dense in $\d$. For  $z\in\d\cap G_\a$, the set $CR(\ell_{-\infty}, \ell_0, \ell_z,\ell_\infty)$ consists of a single element $cr(0,z)$. If $z\in\d$ is non invertible, there exist $z_n\in\d$ which are invertible such that $z_n\to z$ strongly and $\|z_n\|\le \|z\|$. Let us see that $cr(0,z_n)$ converge in some sense to $cr(0,z)$. First note that $cr(0,z_n), cr(0,z)$ are endomorphisms of different submodules. In order to compare them, we can regard them as $\a$-module morphisms of $\a^2$, embedding each module in $\a^2$ using the $\theta$-orthogonal projections $p_{\ell_{z_n}}, p_{\ell_z}$ onto the submodules $\ell_{z_n}, \ell_z$, respectively. For $z'\in\d$, 
$$
p_{\ell_{z'}}(\xx)=(1-|z'|^2)^{-1/2}\theta(\left( \begin{array}{c} 1 \\ z' \end{array} \right), \xx )\left( \begin{array}{c} 1 \\ z' \end{array} \right) (1-|z'|^2)^{-1/2}.
$$
We claim that $cr(0,z_n)p_{\ell_{z_n}}(\xx)\to cr(0,z)p_{\ell_z}(\xx)$ strongly in $\a^2$.
By Proposition \ref{cont sot}, we know that $(1+|z_n|)(1-|z_n|)^{-1}\to (1+|z|)(1-|z|)^{-1}$ strongly. By a similar argument, it also holds that $(1-|z_n|^2)^{-1/2}\to (1-|z|^2)^{-1/2}$ strongly. Also these operators are uniformly bounded. Therefore, using that the product is strongly continuous in bounded sets, our claim follows.  
\end{rem}

\begin{rem}
As a corollary we get that, even if the set $CR(\ell_1,\ell_2,\ell_3,\ell_4)$ may be empty for general $\ell_1,\ell_2,\ell_3,\ell_4$, the particular set $CR(\ell_{-\infty},\ell_0,\ell_z, \ell_\infty)$ is not, and $cr(0,z)$ is a distinguished element of this set. 
\end{rem}

As a first approximation of the deep relationship between the cross ratio and the metric in $\pa^\theta$, we can state the following:
\begin{teo}\label{nice try}
Let $z\in\d$, then 
$$
\frac12\|cr(0,z)\|_{\b(\ell_z)}=d(0,z),
$$
where $\|\ \ \|_{\b(\ell_z)}$ denotes the  norm of operators acting in $\ell_z\subset \a^2$.
\end{teo}
\begin{proof}
Choose for $\ell_z$ the unital basis ${\bf e_z}=\left( \begin{array}{c} (1-z^*z)^{-1/2} \\ z(1-z^*z)^{-1/2} \end{array}\right)$. Then for any $\xx={\bf e_z} a\in\ell_z$,
$$
cr(0,z)\xx={\bf e_z} \log((1+|z|)(1-|z|)^{-1} a,
$$
and thus 
$$
<cr(0,z)\xx, cr(0,z)\xx>=a^* \log((1+|z|)(1-|z|)^{-1}<{\bf e_z},{\bf e_z}>\log((1+|z|)(1-|z|)^{-1}a
$$
$$
=a^*(\log((1+|z|)(1-|z|)^{-1})^2a.
$$
Since $(\log((1+|z|)(1-|z|)^{-1})^2\le \|\log((1+|z|)(1-|z|)^{-1}\|^2$, it follows that 
$$
a^*(\log((1+|z|)(1-|z|)^{-1})^2a\le a^*a \|\log((1+|z|)(1-|z|)^{-1}\|^2,
$$
and therefore
$$
\|<cr(0,z)\xx, cr(0,z)\xx>\|^{1/2}\le  \|\log((1+|z|)(1-|z|)^{-1}\| \|a^*a\|^{1/2}
$$
$$
=\|\log((1+|z|)(1-|z|)^{-1}\| \|\xx\|.
$$
This implies that $\|cr(0,z)\|_{\b(\ell_z)}\le \|\log((1+|z|)(1-|z|)^{-1}\|$. 
Note that 
$$cr(0,z){\bf e_z}={\bf e_z}\log((1+|z|)(1-|z|)^{-1},
$$
so that 
$$
\|cr(0,z){\bf e_z}\|^2=\|<{\bf e_z} \log((1+|z|)(1-|z|)^{-1},{\bf e_z} \log((1+|z|)(1-|z|)^{-1}\|
$$
$$
=\|\log((1+|z|)(1-|z|)^{-1}\|^2,
$$
i.e., 
$$
\|cr(0,z)\|_{\b(\ell_z)}= \|\log((1+|z|)(1-|z|)^{-1}\|=2 d(0,z).
$$
\end{proof}

\section{The coefficient bundle.}

Consider the slight variant of  the commutative diagram in  (\ref{diagrama}):
$$
\xymatrix{
& \k_\theta \ar[ld]_{\hat{\pi_\theta }}\ar[rd]^{\tilde{\pi}_\theta} \\
\q_\rho \ar[rr]^{\simeq}
&& \d ,
}
$$
Recall (from the end of Section 4), that $\q_\rho$ denotes the space of $\theta$-orthogonal rank one projections (considered here the coordinate free  version of $\pa^\theta$). Let us introduce the {\it canonical bundle}
$$
{\bf \xi}\to\q_\rho,
$$
whose fiber over $q\in\q_\rho$ is the module $R(q)$.  This is a fiber bundle of right $\a$-modules, which has a canonical connection. Elements of $\xi$ are pairs $(q,\xx)$, with $q\in\q_\rho$ and $q(\xx)=\xx$. 

Let $q\in\q_\rho$ and $\varphi:R(q)\to R(q)$ a right module endomorphism. Pick a normalized generator $\xx\in R(q)$: $\theta(\xx,\xx)=1$ (i.e., an element of $R(q)$ in $\k_\theta$). Then the endomorphism $\varphi$ is determined by the value $\varphi(\xx)=\xx a$. That is, for any element $\yy=\xx\lambda\in R(q)$, $\varphi(\yy)=\xx a \lambda$. In other words, if we regard $\xx$ as a basis for $R(q)$, $\varphi$ can be expressed as $\lambda\mapsto a\lambda$. We call $a\in\a$  the matrix of $\varphi$ in the basis $\xx$. If  $\xx'$ is another basis of $R(q)$ in $\k_\theta$, then there exists a unitary $u\in\u_\a$ such that $\xx'=\xx u$. If $b$ is the matrix of $\varphi$ in the basis $\xx'$  (i.e., $\varphi(\xx')=\xx' b$), then
$$
\xx u b=\xx' b=\varphi(\xx')=\varphi(\xx u)=\varphi(\xx) u=\xx' au.
$$
Then $ub=au$, which means that the matrix of $\varphi$ in the basis $\xx'$ is  $b=u^*au$ .

This shows that we can regard the set $End(R(q))$ of endomorphisms of $R(q)$, as the set of pairs $(\xx,a)$, where $\xx\in\k_\theta$ with $q(\xx)=\xx$, and $a\in\a$, with the identification
$$
(\xx,a)\sim (\xx u, u^*au) , \ u\in\u_\a.
$$ 
Then, the map $\Gamma\to\q_\rho$ defined by $(q,\varphi)\mapsto q$ for $q\in\q_\rho$ and $\varphi\in End(R(q))$, is a fiber bundle which we call the {\it coefficient bundle}; alternatively $(\xx,a)\sim (\xx u, u^*au)\mapsto [\xx]$. Each fiber of  $\Gamma$ is a  C$^*$-algebra, which is isomorphic to $\a$. The canonical connection of the bundle $\xi$ induces a connection in $\Gamma$ , by the rule:
$$
(D_X \varphi)(\yy)=(D_X \varphi)\yy+\varphi (D_X\yy).
$$
Here, $\varphi$ is a local cross section of $\Gamma$ and $\yy$ is a local cross section of $\xi$. We remark that the connections of $\xi$ and $\Gamma$ are compatible with the action of $\u(\theta)$.

We define  the {\it basic 1-form}. Given $\xx\in\k_\theta$ and $X\in(T\q_\rho)_q$, with $q=\xx \theta(\xx, \ \ )$, put
$$
\kappa_\xx(X)=X\xx.
$$
Note that $X$ is a matrix in $M_2(\a)$, $\theta$-symmetric and $q$-codiagonal: $X\xx\in\xx^{\perp_\theta}=N(q)$. 

Given $X,Y\in(T\q_\rho)_q$, we define the product
$$
\langle X,Y\rangle_\xx=-\theta(\kappa_\xx(X),\kappa_\xx(Y))=-\theta(X\xx,Y\xx).
$$
If the generator $\xx$ is changed for $\xx'=\xx u$, we have
$$
\langle X,Y\rangle_{\xx'}=-\theta(X(\xx u),Y(\xx u))=-u^*\theta(X\xx,Y\xx)u=u^*\langle X,Y\rangle_\xx u.
$$
This means that, given $q=[\xx]=[\xx']$, the product $\langle X,Y\rangle_q$ is well defined as an element of the fiber $\Gamma_q$.

This product is therefore a Hilbertian product in $T\q_\rho$, with values in $\Gamma$. To this efect, note that $T\q_\rho$ is a right module over the bundle $\Gamma$ of coefficients. Indeed, if we fix $\xx\in\k_\theta$ with $q=\xx\theta(\xx,\ \ )$, 
the map $X\mapsto\kappa_\xx(X)=X\xx$ from $(T\q_\rho)_q$ to $N(q)$ is one to one. If we change $\xx$ with $\xx u$, $\kappa_\xx(X)$ changes to $\kappa_{\xx u}(X)=\kappa_\xx(X)u$. If $X\in(T\q_\rho)_q$ and $\varphi\in\Gamma_q$, we define $X\varphi$ as $\kappa_\xx(X\varphi)=X\xx a$, where $\varphi$ is represented by the class of $(\xx,a)$. With this definition we have
$$
\langle X,Y\varphi\rangle_q=\langle X,Y\rangle_q \varphi.
$$

\subsection{The cross ratio, the logarithm and the exponential.}

We have just defined a Hilbertian $\Gamma$-valued structure in $\q_\rho\simeq\pa^\theta$, or, equivalently, in $\d$.
 In particular, the product
$$
\bi \Log_0(z), \Log_0(z)\bd_0
$$
takes values in the set of  endomorphisms of $\ell_0=\left[\left(\begin{array}{c} 1 \\ 0 \end{array}\right)\right]$, where $\Log_0$ is defined in Corollary \ref{64}. It is a positive module endomorphism (given by multiplying the generator ${\bf e_1}$ by a positive element of $a$). Thus, it has a unique positive square root
$ \bi \Log_0(z), \Log_0(z) \bd_0^{1/2}$, which we shall call the {\it $\theta$-modulus}   $mod_0(\Log_0(z))$ of $\Log_0(z)$.  Explicitly, in the generator ${\bf e_1}$, $mod_0(\Log_0(z))$ consists in multiplying the generator by $\log\left( (1+|z|)(1-|z|)^{-1}\right)$.

On the other hand, we saw that, for $z\in\d$,  the endomorphism of $\ell_z$ denoted by $cr(0,z)$, is given by the same coefficient $\log\left( (1+|z|)(1-|z|)^{-1}\right)$, which multiplies the generator $\left(\begin{array}{c} 1 \\ z \end{array}\right)$ of $\ell_z$. 

We shall translate the endomorphism $cr(0,z)$ from $\ell_z$ to $\ell_0$ by means of the parallel transport of $\pa^\theta$, along the geodesic $\delta$, with $\delta(0)=\ell_0$ and $\delta(1)=\ell_z$ (i.e., the same former $\delta$, which under the identification $\d\simeq\pa^\theta$ joins $\delta(0)=0$ and $\delta(1)=z$ in $\d$: $\delta(t)=\omega \tanh(t|\alpha|)$).

The parallel transport of elements of $\d$ (or $\pa^\theta$)  along the geodesic $\delta(t)=e^{t\left(\begin{array}{cc} 0 & \alpha^* \\ \alpha & 0 \end{array}\right)} \cdot 0$, 
where $\alpha$ is, as in Remark 6.1.2  
$$
\alpha=z\sum_{k=0}^\infty \frac{1}{2k+1} (z^*z)^k,
$$
is given by the left action of the invertible matrix \\
$e^{t\left(\begin{array}{cc} 0 & \alpha^* \\ \alpha & 0 \end{array}\right)}:\ell_0\to \ell_{\delta(t)}$. The endomorphism $cr(0,z)$ of $\ell_z$ is transported to $\ell_0$  as 
$$
cr(0,z)_0:=e^{-\left(\begin{array}{cc} 0 & \alpha^* \\ \alpha & 0 \end{array}\right)} cr(0,z) e^{\left(\begin{array}{cc} 0 & \alpha^* \\ \alpha & 0 \end{array}\right)} :\ell_0\to \ell_0.
$$
Our main result (for the origin) is the following:
\begin{teo}\label{el teo}
With the current notation, if $z\in\d$ (or $\ell_z\in\pa^\theta$),
\begin{equation}\label{exp y log}
e^{mod_0(\Log_0(z))}=cr(0,z)_0 \ \hbox{ \rm{or, equivalently, }} \  mod_0(\Log_0(z))=\log(cr(0,z)_0),
\end{equation}
where  the exponential in the left hand equality is the usual exponential of $\a$, $\log$ in the right hand equality is the usual logarithm of $G$, and each endomorphism of $\ell_0$ is identified with its coefficient in the basis ${\bf e_1}=\left(\begin{array}{c} 1 \\ 0 \end{array}\right).$  
\end{teo}
\begin{proof}
Let us prove the first equality. Since we are comparing endomorphisms of $\ell_0$, it suffices to show that they carry the generator ${\bf e_1}$ to the same element in $\a^2$. Note that
$$
cr(0,z)_0({\bf e_1})= e^{-\left(\begin{array}{cc} 0 & \alpha^* \\ \alpha & 0 \end{array}\right)} cr(0,z) e^{\left(\begin{array}{cc} 0 & \alpha^* \\ \alpha & 0 \end{array}\right)}\left(\begin{array}{c} 1 \\ 0 \end{array}\right)
$$
$$
= e^{-\left(\begin{array}{cc} 0 & \alpha^* \\ \alpha & 0 \end{array}\right)} cr(0,z) \left( \begin{array}{c} \cosh(|\alpha|) \\ \omega \sinh(|\alpha|) \end{array}\right);
$$
using that $\omega \tanh(|\alpha|)=\delta(1)=z$ (Lemma \ref{lema62}), this yields
$$
e^{-\left(\begin{array}{cc} 0 & \alpha^* \\ \alpha & 0 \end{array}\right)} cr(0,z) \left( \begin{array}{c} 1 \\ z  \end{array}\right) \cosh(|\alpha|)=e^{-\left(\begin{array}{cc} 0 & \alpha^* \\ \alpha & 0 \end{array}\right)} \left( \begin{array}{c} 1 \\ z  \end{array}\right) (1+|z|)(1-|z|)^{-1} \cosh(|\alpha|).
$$
By the same computation that showed that $e^{\left(\begin{array}{cc} 0 & \alpha^* \\ \alpha & 0 \end{array}\right)} \left( \begin{array}{c} 1 \\ 0  \end{array}\right)=\left( \begin{array}{c} 1 \\ z  \end{array}\right)\cosh(|\alpha|)$ (see the proof of Lemma \ref{lema62}), we have that 
$$
e^{-\left(\begin{array}{cc} 0 & \alpha^* \\ \alpha & 0 \end{array}\right)} \left( \begin{array}{c} 1 \\ z  \end{array}\right)=\left( \begin{array}{c} 1 \\  0  \end{array}\right)\cosh(|\alpha|)^{-1},
$$
i.e.,
$$
cr(0,z)_0({\bf e_1})=\left( \begin{array}{c} 1 \\  0  \end{array}\right)(1+|z|)(1-|z|)^{-1}.
$$
On the other hand, the endomorphism $mod_0(\Log_0(z)$ sends ${\bf e_1}$ to ${\bf e_1}\log\left((1+|z|)(1-|z|)^{-1}\right)$, and thus
$$
e^{mod_0(\Log_0(z))}({\bf e_1})=\left( \begin{array}{c} 1 \\  0  \end{array}\right)(1+|z|)(1-|z|)^{-1}.
$$
\end{proof}
As in Definition \ref{cr general},  let $z_0\ne z_1\in \d$ ($\ell_{z_0}\ne\ell_{z_1}\in\pa^\theta$). Pick $\tilde{g}\in\u(\theta)$ such that $\tilde{g}\cdot 0=z_0$, and denote by $z=\tilde{g}^{-1}\cdot z_1$ as before. Let $\delta$ be the geodesic such that $\delta(0)=0$ and $\delta(1)=z_1$. Then $\delta_{z_0,z_1}=\tilde{g}\cdot \delta$ is the geodesic which joins $z_0$ and $z_1$ at $t=0$ and $t=1$, respectively. Recall that $cr(z_0,z_1)=\tilde{g}cr(z_0,z_1)\tilde{g}^{-1}$. Likewise,
we put
$$
\Log_{z_0}(z_1):=\tilde{g}\Log_0(z)\tilde{g}^{-1} , \  \hbox{ and } \ mod_{z_0}(z_1)=\bi \Log_{z_0}(z_1), \Log_{z_0}(z_1) \bd_{z_0}^{1/2},
$$
where $\bi \varphi, \psi \bd_{z_0}=\tilde{g}\bi \tilde{g}\varphi\tilde{g}^{-1}, \tilde{g}\psi\tilde{g}^{-1}\bd_0 \tilde{g}^{-1}$, and  $\Log_{z_0}$ is the inverse of the exponential $exp_{z_0}:(T\d)_{z_0}\to \d$. It is not difficult to verify that these definitions do not depend on the choice of $\tilde{g}$.

Finally, let us denote by $cr(z_0,z_1)_{z_0}$ the parallel transport of $cr(z_0,z_1)$ from $\ell_{z_1}$ to $\ell_{z_0}$ along the geodesic $\delta_{z_0,z_1}$ (obtained by conjugation as in the case of the origin, by the value at $t=1$ of the one parameter group in $\u(\theta)$ which determines $\delta_{z_0,z_1}$).  The  $\u(\theta)$-covariance  of the data involved enables one to prove the following:
\begin{coro}\label{el coro}
With the current notations,
$$
mod_{z_0}(\Log_{z_0}(z_1))=\log\left(cr(z_0,z_1)_{z_0}\right).
$$
In particular, $\|\Log_{z_0}(z_1)\|_{z_0}=\|\log(cr(z_0,z_1)_{z_0})\|$.
\end{coro}
\bigskip

\begin{figure}[ht]
\centering
\includegraphics[width=0.5\textwidth]{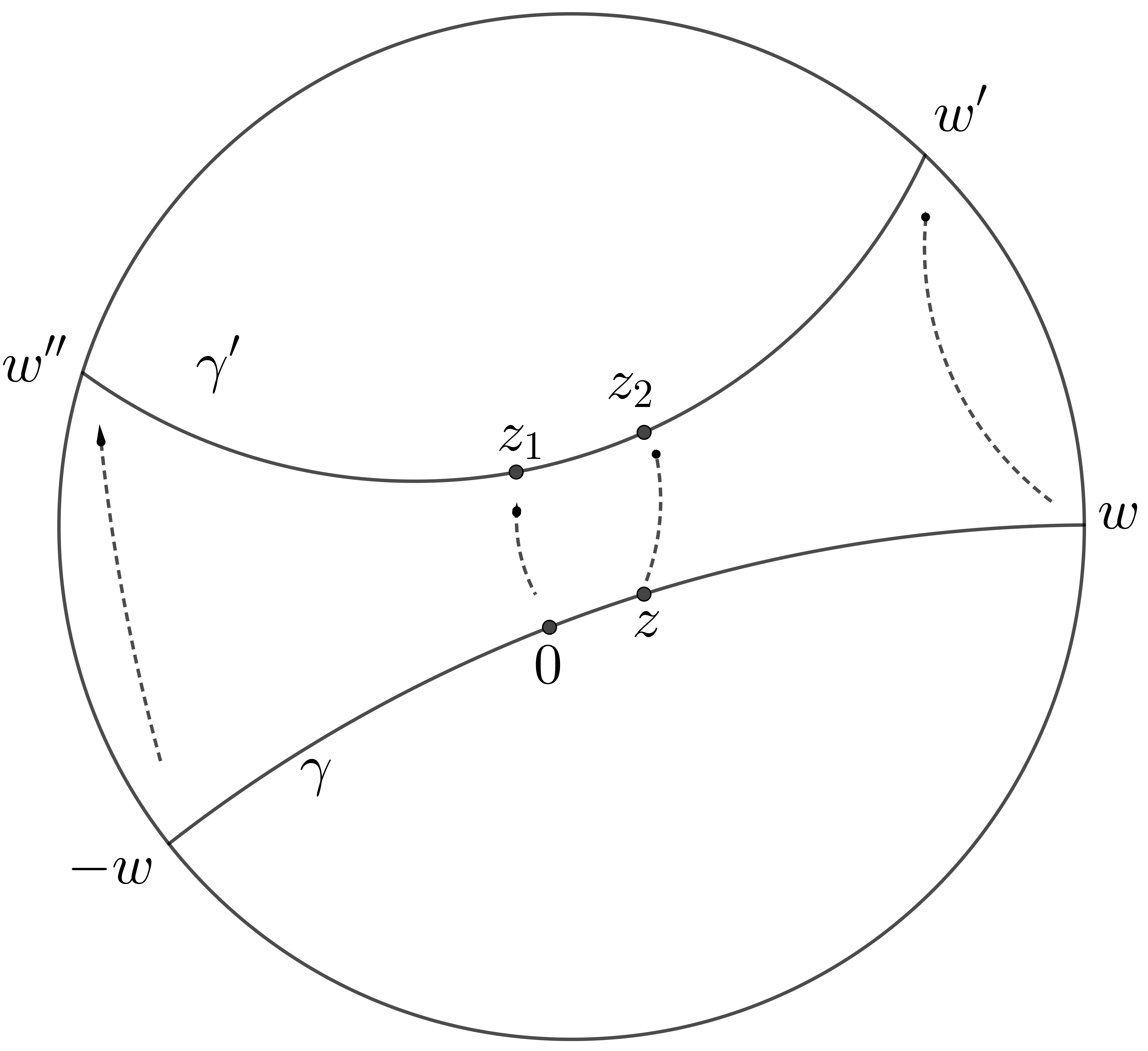}
\end{figure}

\centerline{\bf Figure 2.}
\bigskip

\section{An example.}
Suppose that the algebra $\a$ has a  trace  $\tr$ onto a central subalgebra, that is, there exists a C$^*$- subalgebra $\b\subset Z(\a)$ of the center of $\a$ and a conditional expectation $\tr:\a\to\b$ satisfying $\tr(xy)=\tr(yx)$ for all $x,y\in\a$. This happens, for instance, if  $\a$ is a finite von Neumann algebra. 

A relevant case of this situation is the following. Consider a complex vector bundle $E\to B$ with compact base space $B$, endowed with a Riemannian metric $<e,e'>_b$, $b\in B$, $e,e'\in E_b$ (the fiber of $E$ over $b$). Consider the fiber bundle $End(E)\to B$ of endomorphisms of the vector bundle $E$, and let $\a$ be the algebra $\Gamma(End(E))$ of the continuous global cross sections of $End(E)$. Since each $E_b$ is a (finite dimensional) Hilbert space, $End(E_b)$ is a C$^*$-algebra. The space $\Gamma(End(E))$ of cross sections has therefore the norm $\|\varphi\|=\sup_{b\in B} \|\varphi_b\|$, where  $\varphi_b:E_b\to E_b$ and $\|\varphi_b\|$ is the usual norm of linear  operators.  With this norm, $\a$ is a C$^*$-algebra. The center $Z(\a)$ of this algebra is the space of scalar sections $\lambda$ in $End(E)$ (homotetic in each fiber). The central trace is given by $\tr:\a\to Z(\a)$, $\tr(\sigma)_b=Tr(\sigma_b)$, $b\in B$, with $Tr$ the usual trace of $E_b$. More specifically, $B$ could be a compact manifold, and $E$ the complexification of its tangent bundle, with an Hermitian metric. This case is interesting due to the following observation: in our previous work \cite{tejemas}, we noticed the equivalence, as homogeneous spaces, of the disk $\d$ and the Poincar\'e halfspace $\h$ of the algebra $\a$. This homogeneous space can be thought as the tangent bundle $TG^+$ of the space $G^+$ of positive and invertible elements of $\a$, as explained in \cite{tejemas}. In this context, an element of $\h$ is a pair $(X,a)$ with $a\in G^+$ and $X\in (TG^+)_a$. The element $a\in G^+$ represents a Riemannian  metric in $B$, and a possible vector $X$ (a selfadjoint element of $\a$)  could be the Ricci curvature of the metric $a$. In this manner, the geometry of $\h$ is linked to the deformation of the pairs ({\it Riemannian metric, Ricci curvature}), viewed as elements  of $TG^+$.

Back to the general case (of this example):
$$
\tr \a \to \b\subset Z(\a),
$$
we can define a Hilbertian $\b$-valued inner product, by means of
$$
\langle X, Y \rangle_{\tr, q}= -\tr(\theta(X\xx, Y\yy)).
$$
Indeed, since $\tr$ is tracial, the value of $ -\tr(\theta(X\xx, Y\yy))$ is independent of the choice of $\xx\in\k_\theta$ satisfying $q=\xx\theta(\xx, \ \ )$. On the other hand, $cr(z_0,z_1)$ is an element of $\Gamma_{z_0}$, which has  matrix $a$ in a unital base $\xx\in R(q)$, as explained before. We put $cr(z_0,z_1)_\tr$, for $\tr(a)$. Clearly, $cr(z_0,z_1)_\tr$ does not depend on the basis $\xx$. With these notations, the formula in Corollary \ref{el coro}, can be written
\begin{equation}\label{tracial}
\langle \Log_{z_0} z_1, \Log_{z_0} z_1 \rangle_\tr^{1/2}=\log cr(z_0,z_1)_\tr ,
\end{equation}
which is an identity involving elements in $\b$. 

More specifically, if $\a$ is commutative, we can choose $\tr$ the identity $\a=\b$, and we have

\begin{equation}\label{conmutativa}
|\Log_{z_0} z_1|=\log cr(z_0,z_1),
\end{equation}
as elements in $\a$.

Esteban Andruchow \\
Instituto de Ciencias,  Universidad Nacional de Gral. Sarmiento,
\\
J.M. Gutierrez 1150,  (1613) Los Polvorines, Argentina
\\ 
and Instituto Argentino de Matem\'atica, `Alberto P. Calder\'on', CONICET, 
\\
Saavedra 15 3er. piso,
(1083) Buenos Aires, Argentina.
\\
e-mail: eandruch@ungs.edu.ar

\bigskip

Gustavo Corach\\
Instituto Argentino de Matem\'atica, `Alberto P. Calder\'on', CONICET,
\\
Saavedra 15 3er. piso, (1083) Buenos Aires, Argentina,
\\
and Depto. de Matem\'atica, Facultad de Ingenier\'\i a, Universidad de Buenos Aires, Argentina.
\\
e-mail: gcorach@fi.uba.ar

\bigskip

L\'azaro Recht \\
Departamento de Matem\'atica P y A, 
Universidad Sim\'on Bol\'\i var \\
Apartado 89000, Caracas 1080A, Venezuela  \\
e-mail: recht@usb.ve

\end{document}